\documentclass{amsart}
\usepackage[english]{babel}
\usepackage{amssymb,enumerate,amsmath}
\numberwithin{equation}{section}
\newtheorem{theorem}{Theorem}[section]
\newtheorem{thmA}{Theorem}

\newtheorem{lemma}[theorem]{Lemma}
\newtheorem{proposition}[theorem]{Proposition}
\newtheorem{remark}{Remark}

\newcommand{\mathd}{\mathrm{d}}
\newcommand{\mathe}{\mathrm{e}}
\newcommand{\dt}{\ensuremath{\frac{\partial}{\partial t}}}
\newcommand{\normho}{\ensuremath{| \mathring{h} |}}
\newcommand{\tmscript}[1]{\text{\scriptsize{$#1$}}}
\newcommand{\tmdummy}{$\mbox{}$}
\newcommand{\tmop}[1]{\ensuremath{\operatorname{#1}}}
\newenvironment{enumerateroman}{\begin{enumerate}[\textup{(}i\textup{)}] }{\end{enumerate}}

\begin{document}

\title[Mean Curvature Flow of Arbitrary codimension in $\mathbb{C}\mathbb{P}^m$]{Mean Curvature Flow of Arbitrary codimension in Complex Projective
Spaces}

\author{Li Lei and Hongwei Xu}
\address{Center of Mathematical Sciences \\ Zhejiang University \\ Hangzhou 310027 \\ China}
\email{lei-li@zju.edu.cn; xuhw@cms.zju.edu.cn}

\thanks{Research supported by the National Natural Science Foundation of China, Grant Nos. 11531012, 11371315.}

\begin{abstract}
  In this paper, we investigate the mean curvature flow of submanifolds of
  arbitrary codimension in $\mathbb{C}\mathbb{P}^m$. We prove that if the
  initial submanifold satisfies a pinching condition, then the mean curvature
  flow converges to a round point in finite time, or converges to a totally
  geodesic submanifold as $t \rightarrow \infty$. Consequently, we obtain a
  new differentiable sphere theorem for submanifolds in
  $\mathbb{C}\mathbb{P}^m$. Our work improves the convergence theorem for
  mean curvature flow due to Pipoli and
  Sinestrari {\cite{PiSi2015}}.
\end{abstract}

{\maketitle}

\

\section{Introduction}

Let $F_0 : M^n \rightarrow N^{n + q}$ be an $n$-dimensional submanifold in
the $(n + q)$-dimensional Riemannian manifold $N$. The mean curvature flow
with initial value $F_0$ is a smooth family of immersions $F : M \times [0, T)
\rightarrow N^{n + q}$ satisfying
\begin{equation}
\left\{ \begin{array}{l}
\dt F (x, t) = H (x, t),\\
F (\cdot, 0) = F_0,
\end{array} \right.
\end{equation}
where $H (x, t)$ is the mean curvature vector of the submanifold $M_t = F_t
(M)$, $F_t = F (\cdot, t)$.

In 1984, Huisken {\cite{MR772132}} proved that uniformly convex hypersurfaces
in Euclidean space will converge to a round point along the mean curvature
flow. Afterwards, Huisken obtained convergence results for mean curvature flow
of convex hypersurfaces in Riemannian manifolds {\cite{MR837523}} and pinched
hypersurfaces in spheres {\cite{MR892052}}. In \cite{andrews2002mean}, Andrews
constructed a fully nonlinear parabolic flow of surfaces in the three-sphere, 
and proved an optimal convergence result for this flow.

For higher codimensional submanifolds, Andrews and Baker
\cite{MR2739807} proved an optimal convergence theorem for the mean
curvature flow and an optimal differentiable sphere theorem for
submanifolds in $\mathbb{R}^{n+q}$. Meanwhile, by using the Ricci flow
\cite{BW,Brendle1, MR2449060, Hamilton}, Xu and Gu \cite{XuGu2010}
proved a general differentiable sphere theorem for submanifolds in
space forms independently. Afterwards, Baker \cite{baker2011mean}
and Liu-Xu-Ye-Zhao \cite{MR3078951} proved a sharp convergence
theorem for the mean curvature flow of submanifolds in the space
form $\mathbb{F}^{n+q}(c)$ with $c\neq0$. Later, Liu, Xu and Zhao
{\cite{liu2012mean}} obtained a convergence result for mean curvature
flow of arbitrary codimension in Riemannian manifolds. 
Recently, inspired by the rigidity theory of submanifolds 
\cite{MR1458750, Xu1, MR1241055, xuhwrigidity,Yau}, 
and by developing new techniques, Lei and Xu \cite{lei2015optimal,lei2014sharp,lei2015sharpinS}
verified an optimal convergence theorem for the mean curvature flow of submanifolds in hyperbolic spaces and a new convergence theorem for the mean curvature flow of submanifolds in spheres, which improve the convergence theorems due to  Baker \cite{baker2011mean}, Huisken
\cite{MR892052} and Liu-Xu-Ye-Zhao \cite{MR3078951}. 
For more results on rigidity, sphere and
convergence theorems, we refer the readers to \cite{MR0343217,
ChKa2001, MR0273546,MR3005061,MR1161925, LiWang2014, liu2013mean,
MR0317246, MR0353216, Shiohama, MR1633163, MR0233295, MR2483374,
XG2, MR2550209, Yau, Zhu}.

More recently, Pipoli and Sinestrari {\cite{PiSi2015}} obtain the
following convergence theorem for mean curvature flow of small codimension in
the complex projective space.

\begin{thmA}
	\label{theoPS}Let $F_0 : M^n \rightarrow \mathbb{C}\mathbb{P}^{\frac{n +
	q}{2}}$ be a closed submanifold of dimension $n$ and codimension $q$ in the
	complex projective space with Fubini-Study metric. Suppose either $n
	\geqslant 5$ and $q = 1$, or and $2 \leqslant q < \frac{n - 3}{4}$. Let $F :
	M^n \times [0, T) \rightarrow \mathbb{C}\mathbb{P}^{\frac{n + q}{2}}$ be
	the mean curvature flow with initial value $F_0$. If $F_0$ satisfies
	\[ | h |^2 < \left\{ \begin{array}{ll}
	\frac{1}{n - 1} | H |^2 + 2, & q = 1,\\
	\frac{1}{n - 1} | H |^2 + \frac{n - 3 - 4 q}{n}, & q \geqslant 2,
	\end{array} \right. \]
	then $F_t$ converges to a round point in finite time, or converges to a
	totally geodesic submanifold as $t \rightarrow \infty$. In particular, $M$
	is diffeomorphic to either $\mathbb{S}^n$ or $\mathbb{C}\mathbb{P}^{n /
		2}$.
\end{thmA}

In this paper, we investigate the mean
curvature flow of arbitrary codimensional submanifolds in the complex
projective space, and prove the following theorem.

\begin{theorem}
	\label{theoh}Let $F_0 : M^n \rightarrow \mathbb{C}\mathbb{P}^{\frac{n +
	q}{2}}$ be an n-dimensional closed submanifold in
	$\mathbb{C}\mathbb{P}^{\frac{n + q}{2}}$. Suppose that the dimension and
	codimension satisfy either (i) $q = 1$ and $n \geqslant 3$, (ii) $2
	\leqslant q < n - 4$, or (iii) $q \geqslant n - 4 \geqslant 2$. Let $F : M^n
	\times [0, T) \rightarrow \mathbb{C}\mathbb{P}^{\frac{n + q}{2}}$ be the
	mean curvature flow with initial value $F_0$. If $F_0$ satisfies
	\[ | h |^2 < \left\{ \begin{array}{ll}
	\varphi (| H |^2), & q = 1,\\
	\frac{1}{n - 1} | H |^2 + 2 - \frac{3}{n}, & 2 \leqslant q < n - 4,\\
	\psi (| H |^2), & q \geqslant n - 4,
	\end{array} \right. \]
	then $F_t$ converges to a round point in finite time, or converges to a
	totally geodesic submanifold as $t \rightarrow \infty$. In particular, $M$
	is diffeomorphic to $\mathbb{S}^n$ or $\mathbb{C}\mathbb{P}^{n /
		2}$.
	
	Here the functions $\varphi$ and $\psi$ are defined as (\ref{func1}) and
	(\ref{func2}), and $\varphi (| H |^2)$ and $\psi (| H |^2)$ are given by
	\[ \varphi (| H |^2) = 2 + a_n + \left( b_n + \tfrac{1}{n - 1} \right)
	| H |^2 - \sqrt{b_n^2 | H |^4 + 2 a_n b_n | H |^2}, \]
	\[ \psi (| H |^2) = \frac{9}{n^2 - 3 n - 3} + \frac{n^2 - 3 n}{n^3 - 4
		n^2 + 3} | H |^2 - \tfrac{3 \sqrt{| H |^4 + \frac{2}{n} (n - 1) (n^2 - 3) |
	H |^2 + 9 (n - 1)^2}}{n^3 - 4 n^2 + 3}, \]
	where $a_n = 2 \sqrt{(n^2 - 4 n + 3) b_n}, b_n = \min \left\{ \frac{n - 3}{4
		n - 4}, \frac{2 n - 5}{n^2 + n - 2} \right\}$.
\end{theorem}

\begin{remark}
	By a computation, if $n > 3$, we have $\varphi (x) > \frac{x}{n - 1} + 2$
	for $x \geqslant 0$. Furthermore, if $n \geqslant 3$, we have $\varphi (x)
	> \sqrt{2 (n - 3)}$ for $x \geqslant 0$. Therefore, Theorem
	\ref{theoh} substantially improves Theorem \ref{theoPS}.

	The function $\psi$ satisfies $\psi (x) \geqslant \frac{x}{n}$ for $x
	\geqslant 0$, and the the equality holds if and only if $x = 0$. In
	addition, we have $\lim_{x \rightarrow + \infty} \psi (x) / x = \frac{1}{n -
		1}$.
\end{remark}

Since $\psi (0) = 0$, the pinching condition $| h |^2 < \psi (| H |^2)$
implies that the submanifold has nonzero mean curvature. Thus, if $q
\geqslant n - 4$, it doesn't occur that the mean curvature flow in Theorem
\ref{theoh} converges to a totally geodesic submanifold. When $q \geqslant n -
4 \geqslant 2$, we obtain a refined result under the weakly pinching
condition.

\begin{theorem}
	\label{theo3}Let $F_0 : M^n \rightarrow \mathbb{C}\mathbb{P}^{\frac{n +
	q}{2}}$ be a closed submanifold of dimension $n$ and codimension $q$ in
	$\mathbb{C}\mathbb{P}^{\frac{n + q}{2}}$, where $q \geqslant n - 4
	\geqslant 2$. Let $F : M^n \times [0, T) \rightarrow
	\mathbb{C}\mathbb{P}^{\frac{n + q}{2}}$ be the mean curvature flow with
	initial value $F_0$. If $F_0$ satisfies
	\[ | h |^2 \leqslant \psi (| H |^2), \]
	then one of the following holds:
	\begin{enumerateroman}
        \item $F_t$ is congruent to the totally geodesic
		$\mathbb{R}\mathbb{P}^n$ or $\mathbb{C}\mathbb{P}^{n / 2}$ for each $t
		\in [0, + \infty)$;
		
		\item $F_t$ converges to a round point in finite time. In particular, $M$
		is diffeomorphic to $\mathbb{S}^n$.
	\end{enumerateroman}
\end{theorem}

\section{Notations and formulas}

Let $\mathbb{C}\mathbb{P}^m$ be the $m$-dimensional complex projective
space with the Fubini-Study metric $g_{\tmop{FS}}$. Let $J$ be its complex
structure. We denote by $\bar{\nabla}$ the Levi-Civita connection of
$(\mathbb{C}\mathbb{P}^m, g_{\tmop{FS}})$. Since the Fubini-Study metric is
a K{\"a}hler metric, we have $\bar{\nabla} J = 0$. The curvature tensor
$\bar{R}$ of $\mathbb{C}\mathbb{P}^m$ can be written as
\begin{eqnarray}
\bar{R} (X, Y, Z, W) & = & \langle X, Z \rangle \langle Y, W \rangle -
\langle X, W \rangle \langle Y, Z \rangle \nonumber\\
&  & + \langle X, J Z \rangle \langle Y, J W \rangle - \langle X, J W
\rangle \langle Y, J Z \rangle  \label{CPcurv}\\
&  & + 2 \langle X, J Y \rangle \langle Z, J W \rangle . \nonumber
\end{eqnarray}

Let $(M^n, g)$ be a real $n$-dimensional Riemannian submanifold immersed in
$(\mathbb{C}\mathbb{P}^m, g_{\tmop{FS}})$. Let $q$ be its codimension, i.e.,
$n + q = 2 m$. At a point $p \in M$, let $T_p M$ and $N_p M$ be the tangent
space and normal space, respectively. For a vector in $T_p M \oplus N_p M$, we
denote by $(\cdot)^T$ and $(\cdot)^N$ its projections onto $T_p M$ and $N_p
M$, respectively. We use the same symbol $\nabla$ to represent the connections
of tangent bundle $T M$ and normal bundle $N M$. Denote by $\Gamma (E)$ the
spaces of smooth sections of a vector bundles $E$. For $X, Y \in \Gamma (T
M)$, $\xi \in \Gamma (N M)$, the connections $\nabla$ are given by $\nabla_X Y
= (\bar{\nabla}_X Y)^T$ and $\nabla_X \xi = (\bar{\nabla}_X \xi)^N$. The
second fundamental form of $M$ is defined as $h (X, Y) = (\bar{\nabla}_X
Y)^N$.

Throughout this paper, we shall make the following convention on indices:
\[ 1 \leqslant A, B, C, \cdots \leqslant n + q, \quad 1 \leqslant i, j, k,
\cdots \leqslant n, \quad n + 1 \leqslant \alpha, \beta, \gamma, \cdots
\leqslant n + q. \]
We choose a local orthonormal frame $\{ e_i \}$ for the tangent bundle and a
local orthonormal frame $\{ e_{\alpha} \}$ for the normal bundle. With the
local frame, the components of $h$ are given by $h^{\alpha}_{i j} = \langle h
(e_i, e_j), e_{\alpha} \rangle$. The mean curvature vector is defined as $H =
\sum_{\alpha} H^{\alpha} e_{\alpha}$, where $H^{\alpha} = \sum_i h^{\alpha}_{i
	i}$. Let $\mathring{h} = h - \tfrac{1}{n} H \otimes g$ be the traceless second
fundamental form. We have the relations $\normho^2 = | h |^2 - \frac{1}{n} | H
|^2$ and $\left| \nabla \mathring{h} \right|^2 = | \nabla h |^2 - \frac{1}{n}
| \nabla H |^2$.

We denote by $(J_{A B})$ the matrix of $J$ with respect to the frame $\{ e_A
\}$, i.e.,
\[ J_{A B} = \langle e_A, J e_B \rangle . \]
This matrix satisfies $J_{A B} = - J_{B A}$ and $\sum_B J_{A B} J_{B C} = -
\delta_{A C}$.

At each point  $p\in M$, we define a tensor $P : N_p M \rightarrow T_p M$ by
\[ P \xi = (J \xi)^T \qquad \tmop{for} \quad \xi \in N_p M. \]
Then we have
\[ | P |^2 = \sum_{\alpha} | P e_{\alpha} |^2 \leqslant \sum_{\alpha} |
e_{\alpha} |^2 = q, \]
and
\[ | P |^2 = \sum_{\alpha, i} (J_{i \alpha})^2 = \sum_{A, i} (J_{i A})^2 -
\sum_{i, j} (J_{i j})^2 = n - \sum_{i, j} (J_{i j})^2 . \]

We have the following estimates for the gradient of the second fundamental form.

\begin{lemma}
	\label{dA2}For an $n$-dimensional submanifold in
	$\mathbb{C}\mathbb{P}^{\frac{n + q}{2}}$, we have
	\[ | \nabla h |^2 \geqslant \left\{ \begin{array}{ll}
	\frac{3}{n + 2} | \nabla H |^2 + 2 (n - 1), & q = 1,\\
	\frac{3}{n + 8} | \nabla H |^2 + 2 (n - q) | P |^2, & 2 \leqslant q <
	n,\\
	\frac{3}{n + 8} | \nabla H |^2, & q \geqslant n.
	\end{array} \right. \]
\end{lemma}

\begin{proof}
	Let $S$ be the symmetric part of $\nabla h$, i.e., $S_{i j k}^{\alpha} =
	\frac{1}{3} (\nabla_i h_{j k}^{\alpha} + \nabla_j h_{i k}^{\alpha} +
	\nabla_k h_{i j}^{\alpha})$. Using the same argument as in the proof of Lemma 2.2 in
	{\cite{MR772132}}, we have
	\begin{equation}
	| S |^2 \geqslant \frac{3}{n + 2} \sum_{\alpha, i} \left( \sum_k
	S^{\alpha}_{i k k} \right)^2 . \label{symdelh}
	\end{equation}
	By the Codazzi equation, we have $\sum_k S^{\alpha}_{i k k} = \nabla_i
	H^{\alpha} + 2 \sum_k J_{\alpha k} J_{k i}$. Then we obtain
	\begin{equation}
	\sum_i \left( \sum_k S_{i k k} \right)^2 = | \nabla H |^2 + 4
	\sum_{\alpha, k} \nabla_i H^{\alpha} J_{\alpha k} J_{k i} + 4
	\sum_{\alpha, i} \left( \sum_k J_{\alpha k} J_{k i} \right)^2 .
	\label{trsymdelh}
	\end{equation}
	Using the Codazzi equation again, we get
	\begin{eqnarray}
	| S |^2 & = & \frac{1}{3} | \nabla h |^2 + \frac{2}{3} \sum_{\alpha, i, j,
		k} \nabla_k h^{\alpha}_{i j} \nabla_j h^{\alpha}_{i k} \nonumber\\
	& = & \frac{1}{3} | \nabla h |^2 + \frac{2}{3} \sum_{\alpha, i, j, k}
	\nabla_k h^{\alpha}_{i j} (\nabla_k h^{\alpha}_{i j} + \bar{R}_{\alpha i j
		k}) \nonumber\\
	& = & | \nabla h |^2 + \frac{2}{3} \sum_{\alpha, i, j, k} \nabla_k
	h^{\alpha}_{i j} \bar{R}_{\alpha i j k} \nonumber\\
	& = & | \nabla h |^2 + \frac{2}{3} \sum_{\alpha, i, j, k} (\nabla_i
	h^{\alpha}_{j k} + \bar{R}_{\alpha j k i}) \bar{R}_{\alpha i j k} 
	\label{normsymdelh}\\
	& = & | \nabla h |^2 + \frac{2}{3} \sum_{\alpha, i, j, k} \bar{R}_{\alpha
		j k i} \bar{R}_{\alpha i j k} \nonumber\\
	& = & | \nabla h |^2 - 2 \sum_{\alpha, i, j, k} [J_{\alpha j} J_{j i}
	J_{\alpha k} J_{k i} + (J_{\alpha i})^2 (J_{j k})^2] \nonumber\\
	& = & | \nabla h |^2 - 2 \sum_{\alpha, i} \left( \sum_k J_{\alpha k} J_{k
		i} \right)^2 - 2 | P |^2 (n - | P |^2) . \nonumber
	\end{eqnarray}
	From (\ref{symdelh}),(\ref{trsymdelh}) and (\ref{normsymdelh}), we obtain
	\begin{eqnarray}
	| \nabla h |^2 & \geqslant & \frac{1}{n + 2} \left[ 3 | \nabla H |^2 + 12
	\sum_{\alpha, i, k} \nabla_i H^{\alpha} J_{\alpha k} J_{k i} + 2 (n + 8)
	\sum_{\alpha, i} \left( \sum_k J_{\alpha k} J_{k i} \right)^2 \right]
	\nonumber\\
	&  & + 2 | P |^2 (n - | P |^2) .  \label{dh23n2}
	\end{eqnarray}
	If $q = 1$, we have $\sum_k J_{\alpha k} J_{k i} = 0$ and $| P |^2 = 1$.
	Thus (\ref{dh23n2}) becomes
	\[ | \nabla h |^2 \geqslant \frac{3}{n + 2} | \nabla H |^2 + 2 (n - 1) . \]
	If $q \geqslant 2$, from (\ref{dh23n2}) we get
	\begin{eqnarray*}
		| \nabla h |^2 & \geqslant & \frac{2}{n + 2} \sum_{\alpha, i} \left( 3
		\sqrt{\frac{1}{n + 8}} \nabla_i H^{\alpha} + \sqrt{n + 8} \sum_k J_{\alpha
	k} J_{k i} \right)^2\\
		&  & + \frac{3}{n + 8} | \nabla H |^2 + 2 | P |^2 (n - | P |^2) .
	\end{eqnarray*}
	
\end{proof}

Let $F : M^n \times [0, T) \rightarrow \mathbb{C}\mathbb{P}^m$ be a mean
curvature flow in a complex projective space. For a fixed $t$, letting $F_t =
F (\cdot, t)$, then $F_t : M^n \rightarrow \mathbb{C}\mathbb{P}^m$ is a
Riemannian submanifold in $\mathbb{C}\mathbb{P}^m$. We denote by $M_t$ the
submanifold at time $t$. Following {\cite{MR2739807,PiSi2015}}, we have the
evolution equations below.
\begin{eqnarray*}
	\dt | h |^2 & = & \Delta | h |^2 - 2 | \nabla h |^2\\
	&  & + 2 \sum_{\alpha, \beta} \left( \sum_{i, j} h^{\alpha}_{i j}
	h^{\beta}_{i j} \right)^2 + 2 \sum_{i, j, \alpha, \beta} \left( \sum_k
	(h^{\alpha}_{i k} h^{\beta}_{j k} - h^{\beta}_{i k} h^{\alpha}_{j k})
	\right)^2\\
	&  & + 2 \sum_{\alpha, \beta, i, j, k} h^{\alpha}_{i j} h^{\beta}_{i j}
	\bar{R}_{\alpha k \beta k} + 8 \sum_{\alpha, \beta, i, j, k}
	\mathring{h}^{\alpha}_{i k} \mathring{h}^{\beta}_{j k} \bar{R}_{\alpha \beta
		i j}\\
	&  & + 4 \sum_{\alpha, i, j, k, l} (h^{\alpha}_{i k} h^{\alpha}_{j l}
	\bar{R}_{i j k l} - h^{\alpha}_{i k} h^{\alpha}_{j k} \bar{R}_{i l j l}),
\end{eqnarray*}
\[ \dt | H |^2 = \Delta | H |^2 - 2 | \nabla H |^2 + 2 \sum_{i, j} \left(
\sum_{\alpha} H^{\alpha} h^{\alpha}_{i j} \right)^2 + 2 \sum_{\alpha,
	\beta, k} H^{\alpha} H^{\beta} \bar{R}_{\alpha k \beta k} . \]
From (\ref{CPcurv}), these evolution equations can be written as

\begin{lemma}
	\label{evoinCP}For mean curvature flow $F : M^n \times [0, T) \rightarrow
	\mathbb{C}\mathbb{P}^m$, we have
	\begin{enumerateroman}
		\item $\dt | h |^2 = \Delta | h |^2 - 2 | \nabla h |^2 - 2 n | h |^2 + 4 |
		H |^2 + 2 R_1 + 2 S_1$,
		
		\item $\dt | H |^2 = \Delta | H |^2 - 2 | \nabla H |^2 + 2 n | H |^2 + 2
		R_2 + 6 S_2$,
	\end{enumerateroman}
	where
	\[ R_1 = \sum_{\alpha, \beta} \left( \sum_{i, j} h^{\alpha}_{i j}
	h^{\beta}_{i j} \right)^2 + \sum_{i, j, \alpha, \beta} \left( \sum_k
	(h^{\alpha}_{i k} h^{\beta}_{j k} - h^{\beta}_{i k} h^{\alpha}_{j k})
	\right)^2, \]
	\[ R_2 = \sum_{i, j} \left( \sum_{\alpha} H^{\alpha} h^{\alpha}_{i j}
	\right)^2, \]
	\begin{eqnarray*}
		S_1 & = & 3 \sum_{i, j, k} \left( \sum_{\alpha} h^{\alpha}_{i j} J_{k
	\alpha} \right)^2 + 4 \sum_{\alpha, \beta, i, j, k}
		\mathring{h}^{\alpha}_{i k} \mathring{h}^{\beta}_{j k} (J_{i \alpha} J_{j
	\beta} - J_{i \beta} J_{j \alpha})\\
		&  & + 6 \sum_{\alpha, i, j, k, l} \left( \mathring{h}^{\alpha}_{i j}
		\mathring{h}^{\alpha}_{k l} J_{i l} J_{j k} - \mathring{h}^{\alpha}_{i k}
		\mathring{h}^{\alpha}_{j k} J_{i l} J_{j l} \right) + 8 \sum_{\alpha,
	\beta, i, j, k} \mathring{h}^{\alpha}_{i k} \mathring{h}^{\beta}_{j k}
		J_{\alpha \beta} J_{i j},
	\end{eqnarray*}
	\[ S_2 = \sum_k \left( \sum_{\alpha} H^{\alpha} J_{k \alpha} \right)^2 . \]
	
\end{lemma}

To do computations involving $J_{A B}$, we present the following well-known
property of the skew-symmetric matrix.

\begin{proposition}
	Let $A$ be a real skew-symmetric matrix. Then there exists an orthogonal
	matrix $C$, such that $C^{- 1} A C$ takes the following form
	\begin{equation}
	\left(\begin{array}{cccccccc}
	0 & \lambda_1 &  &  &  &  &  & \\
	- \lambda_1 & 0 &  &  &  &  &  & \\
	&  & 0 & \lambda_3 &  &  &  & \\
	&  & - \lambda_3 & 0 &  &  &  & \\
	&  &  &  & 0 & \lambda_5 &  & \\
	&  &  &  & - \lambda_5 & 0 &  & \\
	&  &  &  &  &  & \ddots & \\
	&  &  &  &  &  &  & \ddots
	\end{array}\right) . \label{antisym}
	\end{equation}
\end{proposition}

We use a notation
\[ \tilde{i} = \left\{ \begin{array}{ll}
i + 1, & i \;  \tmop{is}\;  \tmop{odd},\\
i - 1, & i \;  \tmop{is}\;  \tmop{even} .
\end{array} \right. \]
If a matrix $(a_{i j})$ takes the form of (\ref{antisym}), then $a_{i j} = 0$
for all $j \neq \tilde{i}$.

\section{Preservation of curvature pinching}

\subsection{The case of $q = 1$}\

For the mean curvature flow of hypersurfaces in $\mathbb{C}\mathbb{P}^m$,
the evolution equation equations in Lemma \ref{evoinCP} become
\begin{eqnarray*}
	\dt \normho^2 & = & \Delta \normho^2 - 2 \left| \nabla \mathring{h}
	\right|^2 - 2 n \normho^2 + 2 \normho^2 | h |^2 + 6 \normho^2\\
	&  & + 12 \sum_{i, j, k, l} \left( \mathring{h}_{i j} \mathring{h}_{k l}
	J_{i l} J_{j k} - \mathring{h}_{i k} \mathring{h}_{j k} J_{i l} J_{j l}
	\right),
\end{eqnarray*}
\[ \dt | H |^2 = \Delta | H |^2 - 2 | \nabla H |^2 + 2 n | H |^2 + 2 | H |^2 |
h |^2 + 6 | H |^2 . \]
We choose a orthonormal frame $\{ e_i \}$ such that the matrix $(J_{i j})$
takes the form of (\ref{antisym}). Thus
\begin{eqnarray*}
	&  & \sum_{i, j, k, l} \left( \mathring{h}_{i j} \mathring{h}_{k l} J_{i l}
	J_{j k} - \mathring{h}_{i k} \mathring{h}_{j k} J_{i l} J_{j l} \right)\\
	& = & \sum_{i, k} \left( - \mathring{h}_{i \tilde{k}} \mathring{h}_{k
		\tilde{i}} J_{i \tilde{i}} J_{k \tilde{k}} - \left(
	\mathring{h}^{\alpha}_{\tilde{i} k} J_{i \tilde{i}} \right)^2 \right)\\
	& = & - \frac{1}{2} \sum_{i, k, \alpha} \left( \mathring{h}_{i \tilde{k}}
	J_{k \tilde{k}} + \mathring{h}^{\alpha}_{\tilde{i} k} J_{i \tilde{i}}
	\right)^2 \leqslant 0.
\end{eqnarray*}
So, we get
\begin{equation}
\dt \normho^2 \leqslant \Delta \normho^2 - 2 \left| \nabla \mathring{h}
\right|^2 + 2 \normho^2 (| h |^2 - n + 3). \label{parahoh}
\end{equation}

For a real number $\varepsilon \in [0, 1]$, we define a function
$\varphi_{\varepsilon} : [0, + \infty) \rightarrow \mathbb{R}$ by
\begin{equation}
\varphi_{\varepsilon} (x) := d_{\varepsilon} + c_{\varepsilon} x -
\sqrt{b^2 x^2 + 2 a b x + e}, \label{func1}
\end{equation}
where $a = 2 \sqrt{(n^2 - 4 n + 3) b}, b = \min \left\{ \frac{n - 3}{4 n - 4},
\frac{2 n - 5}{n^2 + n - 2} \right\}, c_{\varepsilon} = b + \tfrac{1}{n - 1 +
	\varepsilon}, d_{\varepsilon} = 2 - 2 \varepsilon + a, e =
\sqrt{\varepsilon}$. We define $\varphi = \varphi_0$.

Let $\mathring{\varphi}_{\varepsilon} (x) = \varphi_{\varepsilon} (x) -
\frac{x}{n}$. The following lemma will be proven in the Appendix.

\begin{lemma}
	\label{funcphy}For sufficiently small $\varepsilon$, the function
	$\mathring{\varphi}_{\varepsilon}$ satisfies
	\begin{enumerateroman}
		\item $2 x \mathring{\varphi}_{\varepsilon}'' (x) +
		\mathring{\varphi}_{\varepsilon}' (x) < \frac{2 (n - 1)}{n (n + 2)}$,
		
		\item $\mathring{\varphi}_{\varepsilon} (x) (\varphi_{\varepsilon} (x) - n
		+ 3) - x \mathring{\varphi}_{\varepsilon}' (x) (\varphi_{\varepsilon} (x)
		+ n + 3) < 2 (n - 1)$,
		
		\item $\mathring{\varphi}_{\varepsilon} (x) - x
		\mathring{\varphi}_{\varepsilon}' (x) > 1$.
	\end{enumerateroman}
\end{lemma}

Suppose that $M_0$ is an $n$-dimensional closed hypersurface in
$\mathbb{C}\mathbb{P}^{\frac{n + 1}{2}}$ satisfying $| h |^2 < \varphi (| H
|^2)$. Let $F : M^n \times [0, T) \rightarrow \mathbb{C}\mathbb{P}^{\frac{n
		+ 1}{2}}$ be a mean curvature flow with initial value $M_0$. We will show that
the pinching condition is preserved along the flow. For convenience, we denote
$\mathring{\varphi}_{\varepsilon} (| H |^2)$,
$\mathring{\varphi}'_{\varepsilon} (| H |^2)$,
$\mathring{\varphi}_{\varepsilon}'' (| H |^2)$ by
$\mathring{\varphi}_{\varepsilon}$, $\mathring{\varphi}_{\varepsilon}'$,
$\mathring{\varphi}_{\varepsilon}''$, respectively.

\begin{theorem}
	\label{preh}If the initial value $M_0$ satisfies $| h |^2 < \varphi (| H
	|^2)$, then there exists a small positive number $\varepsilon$, such that
	for all $t \in [0, T)$, we have $| h |^2 < \varphi_{\varepsilon} (| H |^2) -
	\varepsilon | H |^2 - \varepsilon$.
\end{theorem}

\begin{proof}
	Since $M_0$ is compact, there exists a small positive number $\varepsilon$,
	such that $M_0$ satisfies $\normho^2 < \mathring{\varphi}_{\varepsilon}$.
	
	From Lemma \ref{funcphy} (i), we have
	\begin{eqnarray}
	\left( \dt - \Delta \right) \mathring{\varphi}_{\varepsilon} & = & - 2
	\left( \mathring{\varphi}_{\varepsilon}' + 2
	\mathring{\varphi}_{\varepsilon}'' \cdot | H |^2 \right) | \nabla H |^2 +
	2 \mathring{\varphi}_{\varepsilon}' \cdot | H |^2 (| h |^2 + n + 3)
	\nonumber\\
	& \geqslant & - \frac{4 (n - 1)}{n (n + 2)} | \nabla H |^2 + 2
	\mathring{\varphi}_{\varepsilon}' \cdot | H |^2 (| h |^2 + n + 3) . 
	\label{parafh}
	\end{eqnarray}
	Let $U = \normho^2 - \mathring{\varphi}_{\varepsilon}$. We obtain
	\begin{eqnarray*}
		\frac{1}{2} \left( \dt - \Delta \right) U & \leqslant & - \left| \nabla
		\mathring{h} \right|^2 + \frac{2 (n - 1)}{n (n + 2)} | \nabla H |^2\\
		&  & + \normho^2 (| h |^2 - n + 3) - \mathring{\varphi}_{\varepsilon}'
		\cdot | H |^2 (| h |^2 + n + 3) .
	\end{eqnarray*}
	By Lemma \ref{dA2}, we have
	\[ - \left| \nabla \mathring{h} \right|^2 + \frac{2 (n - 1)}{n (n + 2)} |
	\nabla H |^2 \leqslant - 2 (n - 1) . \]
	Thus, at the points where $U = 0$, we get
	\[ \frac{1}{2} \left( \dt - \Delta \right) U \leqslant - 2 (n - 1) +
	\mathring{\varphi}_{\varepsilon} (\varphi_{\varepsilon} - n + 3) -
	\mathring{\varphi}_{\varepsilon}' \cdot | H |^2 (\varphi_{\varepsilon} +
	n + 3) < 0. \]
	Applying the maximum principle, we obtain $U < 0$ for all $t \in [0, T)$.
	
	By choosing a suitable small $\varepsilon$, we complete the proof of
	Theorem \ref{preh}.
\end{proof}

Let $f_{\sigma} = \normho^2 / \mathring{\varphi}_{\varepsilon}^{1 - \sigma}$,
where $\sigma \in (0, 1)$ is a positive constant. Then we have

\begin{lemma}
	\label{ptf1}If $M_0$ satisfies $| h |^2 < \varphi (| H |^2)$, then there
	exists a small positive number $\varepsilon$, such that the following
	inequality holds along the mean curvature flow.
	\[ \dt f_{\sigma} \leqslant \Delta f_{\sigma} +
	\frac{2}{\mathring{\varphi}_{\varepsilon}} | \nabla f_{\sigma} | \left|
	\nabla \mathring{\varphi}_{\varepsilon} \right| - \frac{2 \varepsilon
		f_{\sigma}}{n \normho^2} \left| \nabla \mathring{h} \right|^2 + 2 \sigma
	| h |^2 f_{\sigma} - \frac{\varepsilon}{n} f_{\sigma} . \]
\end{lemma}

\begin{proof}
	By a straightforward calculation, we have
	\begin{eqnarray*}
		\left( \dt - \Delta \right) f_{\sigma} & = & f_{\sigma} \left[
		\frac{1}{\normho^2} \left( \dt - \Delta \right) \normho^2 - \frac{1 -
			\sigma}{\mathring{\varphi}_{\varepsilon}} \left( \dt - \Delta \right)
		\mathring{\varphi}_{\varepsilon} \right]\\
		&  & + 2 (1 - \sigma) \frac{\left\langle \nabla f_{\sigma}, \nabla
			\mathring{\varphi}_{\varepsilon}
			\right\rangle}{\mathring{\varphi}_{\varepsilon}} - \sigma (1 - \sigma)
		f_{\sigma}  \frac{\left| \nabla \mathring{\varphi}_{\varepsilon}
			\right|^2}{\left| \mathring{\varphi}_{\varepsilon} \right|^2} .
	\end{eqnarray*}
	Using (\ref{parahoh}) and (\ref{parafh}), we have
	\begin{eqnarray*}
		\left( \dt - \Delta \right) f_{\sigma} & \leqslant & 2 f_{\sigma} \left[ -
		\frac{\left| \nabla \mathring{h} \right|^2}{\normho^2} + \frac{2 (n -
			1)}{n (n + 2)} \frac{| \nabla H |^2}{\mathring{\varphi}_{\varepsilon}}
		\right]\\
		&  & + 2 f_{\sigma} \left[ | h |^2 - n + 3 - (1 - \sigma)
		\frac{\mathring{\varphi}_{\varepsilon}' \cdot | H
			|^2}{\mathring{\varphi}_{\varepsilon}} (| h |^2 + n + 3) \right]\\
		&  & + \frac{2}{\mathring{\varphi}} | \nabla f_{\sigma} | \left| \nabla
		\mathring{\varphi} \right| .
	\end{eqnarray*}
	From Lemma \ref{dA2} and Theorem \ref{preh}, we have
	\begin{eqnarray*}
		- \frac{\left| \nabla \mathring{h} \right|^2}{\normho^2} + \frac{2 (n -
			1)}{n (n + 2)} \frac{| \nabla H |^2}{\mathring{\varphi}_{\varepsilon}} &
		\leqslant & - \frac{\left| \nabla \mathring{h} \right|^2}{\normho^2} +
		\frac{\left| \nabla \mathring{h} \right|^2 - 2 (n -
			1)}{\mathring{\varphi}_{\varepsilon}}\\
		& \leqslant & \frac{\normho^2 -
			\mathring{\varphi}_{\varepsilon}}{\normho^2
			\mathring{\varphi}_{\varepsilon}} \left| \nabla \mathring{h} \right|^2 -
		\frac{2 (n - 1)}{\mathring{\varphi}_{\varepsilon}}\\
		& \leqslant & - \varepsilon \frac{| H |^2 + 1}{\normho^2
			\mathring{\varphi}_{\varepsilon}} \left| \nabla \mathring{h} \right|^2 -
		\frac{2 (n - 1)}{\mathring{\varphi}_{\varepsilon}}\\
		& \leqslant & - \frac{\varepsilon}{n \normho^2} \left| \nabla
		\mathring{h} \right|^2 - \frac{2 (n -
			1)}{\mathring{\varphi}_{\varepsilon}} .
	\end{eqnarray*}
	From Lemma \ref{funcphy} (ii) and (iii), we have
	\begin{eqnarray*}
		&  & | h |^2 - n + 3 - (1 - \sigma)
		\frac{\mathring{\varphi}_{\varepsilon}' \cdot | H
			|^2}{\mathring{\varphi}_{\varepsilon}} (| h |^2 + n + 3)\\
		& = & \frac{1 - \sigma}{\mathring{\varphi}_{\varepsilon}} \left[ \left(
		\mathring{\varphi}_{\varepsilon} - \mathring{\varphi}_{\varepsilon}' \cdot
		| H |^2 \right) | h |^2 - \mathring{\varphi}_{\varepsilon}' \cdot | H |^2
		(n + 3) \right] - n + 3 + \sigma | h |^2\\
		& \leqslant & \frac{1 - \sigma}{\mathring{\varphi}_{\varepsilon}} \left[
		\left( \mathring{\varphi}_{\varepsilon} -
		\mathring{\varphi}_{\varepsilon}' \cdot | H |^2 \right)
		(\varphi_{\varepsilon} - \varepsilon | H |^2 - \varepsilon) -
		\mathring{\varphi}_{\varepsilon}' \cdot | H |^2 (n + 3) \right] - n + 3 +
		\sigma | h |^2\\
		& = & \frac{1 - \sigma}{\mathring{\varphi}_{\varepsilon}} \left[ \left(
		\mathring{\varphi}_{\varepsilon} - \mathring{\varphi}_{\varepsilon}' \cdot
		| H |^2 \right) \varphi_{\varepsilon} - \mathring{\varphi}_{\varepsilon}'
		\cdot | H |^2 (n + 3) \right] - n + 3 + \sigma | h |^2\\
		&  & - \frac{(1 - \sigma) \varepsilon}{\mathring{\varphi}_{\varepsilon}}
		\left( \mathring{\varphi}_{\varepsilon} -
		\mathring{\varphi}_{\varepsilon}' \cdot | H |^2 \right) (| H |^2 + 1)\\
		& \leqslant & (1 - \sigma) \left[ n - 3 + \frac{2 (n -
			1)}{\mathring{\varphi}_{\varepsilon}} \right] - n + 3 + \sigma | h |^2 -
		\frac{(1 - \sigma) \varepsilon}{\mathring{\varphi}_{\varepsilon}} (| H |^2
		+ 1)\\
		& \leqslant & \sigma | h |^2 + \frac{2 (n -
			1)}{\mathring{\varphi}_{\varepsilon}} - \frac{\varepsilon}{2 n} .
	\end{eqnarray*}
	This completes the proof of the lemma.
\end{proof}

\subsection{The case of $2 \leqslant q < n - 4$}\

At a fixed point $p \in M$, we always choose the orthonormal frame $\{
e_{\alpha} \}$ for $N_p M$ such that $H = | H | e_{n + 1}$. For fixed
$\alpha$, let $\left| \mathring{h}^{\alpha} \right|^2 = \sum_{i, j} \left(
\mathring{h}^{\alpha}_{i j} \right)^2$. Set $\rho_1 = \left| \mathring{h}^{n +
	1} \right|^2$, $\rho_2 = \sum_{\alpha > n + 1} \left| \mathring{h}^{\alpha}
\right|^2$, $\theta_1 = | P e_{n + 1} |^2$ and $\theta_2 = \sum_{\alpha > n +
	1} | P e_{\alpha} |^2$. Notice that $\theta_1 \leqslant 1$ and $\theta_2
\leqslant q - 1$.

\begin{lemma}
	\label{nqineq}{\tmdummy}
	
	\begin{enumerateroman}
		\item  $R_1 \leqslant | h |^4 - \frac{2}{n} \rho_2 | H |^2 + 2 \rho_1
		\rho_2 + \frac{1}{2} \rho_2^2$,
		
		\item $R_2 = | H |^2 (| h |^2 - \rho_2)$,
		
		\item $S_1 \leqslant \frac{3}{n} S_2 + 3 \normho^2 + 8 \sqrt{\theta_2
			\rho_1 \rho_2} + 4 \theta_2 \rho_2$.
	\end{enumerateroman}
\end{lemma}

\begin{proof}
	The estimates of $R_1$ and $R_2$ are similar to that in {\cite{MR2739807}}.
	We choose an orthonormal frame $\{ e_i \}$ for the tangent space, such that
	$\mathring{h}^{n + 1}_{i j} = \mathring{\lambda}_i \delta_{i j}$. Then we
	have
	\begin{eqnarray}
	R_1 & = & \rho_1^2 + \frac{2}{n} \rho_1 | H |^2 + \frac{1}{n^2} | H |^4
	\nonumber\\
	&  & + 2 \sum_{\alpha > n + 1} \left( \sum_i \mathring{\lambda}_i 
	\mathring{h}^{\alpha}_{i i} \right)^2 + 2 \sum_{\tmscript{\begin{array}{c}
			\alpha > n + 1\\
			i \neq j
			\end{array}}} \left( \left( \mathring{\lambda}_i - \mathring{\lambda}_j
	\right) \mathring{h}^{\alpha}_{i j} \right)^2 \\
	&  & + \sum_{\alpha, \beta > n + 1} \left( \sum_{i, j}
	\mathring{h}^{\alpha}_{i j}  \mathring{h}^{\beta}_{i j} \right)^2 +
	\sum_{\tmscript{\begin{array}{c}
			\alpha, \beta > n + 1\\
			i, j
			\end{array}}} \left( \sum_k \left( \mathring{h}^{\alpha}_{i k} 
	\mathring{h}^{\beta}_{j k} - \mathring{h}^{\alpha}_{j k} 
	\mathring{h}^{\beta}_{i k} \right) \right)^2 . \nonumber
	\end{eqnarray}
	Using the Cauchy-Schwarz inequality, we get
	\begin{eqnarray*}
		\left( \sum_i \mathring{\lambda}_i  \mathring{h}^{\alpha}_{i i} \right)^2
		+ \sum_{i \neq j} \left( \left( \mathring{\lambda}_i -
		\mathring{\lambda}_j \right) \mathring{h}^{\alpha}_{i j} \right)^2 &
		\leqslant & \rho_1 \sum_i \left(  \mathring{h}^{\alpha}_{i i} \right)^2 +
		2 \sum_{i \neq j} \left( \mathring{\lambda}_i^2 + \mathring{\lambda}_j^2
		\right) \left( \mathring{h}^{\alpha}_{i j} \right)^2\\
		& \leqslant & \rho_1 \sum_i \left(  \mathring{h}^{\alpha}_{i i} \right)^2
		+ 2 \rho_1 \sum_{i \neq j} \left( \mathring{h}^{\alpha}_{i j} \right)^2\\
		& \leqslant & 2 \rho_1 \left|  \mathring{h}^{\alpha} \right|^2 .
	\end{eqnarray*}
	It follows from Theorem 1 of {\cite{MR1161925}} that
	\[ \sum_{\alpha, \beta > n + 1} \left( \sum_{i, j} \mathring{h}^{\alpha}_{i
		j}  \mathring{h}^{\beta}_{i j} \right)^2 +
	\sum_{\tmscript{\begin{array}{c}
			\alpha, \beta > n + 1\\
			i, j
			\end{array}}} \left( \sum_k \left( \mathring{h}^{\alpha}_{i k} 
	\mathring{h}^{\beta}_{j k} - \mathring{h}^{\alpha}_{j k} 
	\mathring{h}^{\beta}_{i k} \right) \right)^2 \leqslant \frac{3}{2}
	\rho_2^2 . \]
	Thus we obtain
	\begin{eqnarray*}
		R_1 & \leqslant & \rho_1^2 + \frac{2}{n} \rho_1 | H |^2 + \frac{1}{n^2} |
		H |^4 + 4 \rho_1 \rho_2 + \frac{3}{2} \rho_2^2\\
		& = & | h |^4 - \frac{2}{n} \rho_2 | H |^2 + 2 \rho_1 \rho_2 +
		\frac{1}{2} \rho_2^2 .
	\end{eqnarray*}
	$R_2$ can be written as
	\begin{equation}
	R_2 = \sum_{i, j} (| H | h^{n + 1}_{i j})^2 = | H |^2 (| h |^2 - \rho_2) .
	\label{R2equa}
	\end{equation}
	Next, we have
	\begin{eqnarray}
	&  & \frac{1}{n} \sum_k \left( \sum_{\alpha} H^{\alpha} J_{k \alpha}
	\right)^2 + \sum_{i, j, k} \left( \sum_{\alpha} \mathring{h}^{\alpha}_{i
		j} J_{k \alpha} \right)^2 \nonumber\\
	& = & \frac{S_2}{n} + \sum_{i, j} \left| P \mathring{h} (e_i, e_j)
	\right|^2  \label{s2nP}\\
	& = & \frac{S_2}{n} + \left| P \mathring{h} \right|^2 . \nonumber
	\end{eqnarray}
	Choose an orthonormal frame $\{ e_i \}$ such that the matrix $(J_{i j})$
	takes the form of (\ref{antisym}). Thus
	\begin{eqnarray}
	&  & 6 \sum_{\alpha, i, j, k, l} (h^{\alpha}_{i j} h^{\alpha}_{k l} J_{i
		l} J_{j k} - h^{\alpha}_{i k} h^{\alpha}_{j k} J_{i l} J_{j l}) + 8
	\sum_{\alpha, \beta, i, j, k} h^{\alpha}_{i k} h^{\beta}_{j k} J_{\alpha
		\beta} J_{i j} \nonumber\\
	& = & 6 \sum_{i, k, \alpha} \left( - \mathring{h}^{\alpha}_{i \tilde{k}}
	\mathring{h}^{\alpha}_{k \tilde{i}} J_{i \tilde{i}} J_{k \tilde{k}} -
	\left( \mathring{h}^{\alpha}_{\tilde{i} k} J_{i \tilde{i}} \right)^2
	\right) + 8 \sum_{i, k, \alpha, \beta} \mathring{h}^{\alpha}_{i k}
	\mathring{h}^{\beta}_{\tilde{i} k} J_{\alpha \beta} J_{i \tilde{i}}
	\nonumber\\
	& = & - 3 \sum_{i, k, \alpha} \left( \mathring{h}^{\alpha}_{i \tilde{k}}
	J_{k \tilde{k}} + \mathring{h}^{\alpha}_{\tilde{i} k} J_{i \tilde{i}}
	\right)^2 - 4 \sum_{i, k, \alpha} \left[ \left( \mathring{h}^{\alpha}_{i
		\tilde{k}} J_{k \tilde{k}} + \mathring{h}^{\alpha}_{\tilde{i} k} J_{i
		\tilde{i}} \right) \sum_{\beta} \mathring{h}^{\beta}_{i k} J_{\alpha
		\beta} \right] \nonumber\\
	& \leqslant & \frac{4}{3} \sum_{i, k, \alpha} \left( \sum_{\beta}
	\mathring{h}^{\beta}_{i k} J_{\alpha \beta} \right)^2  \label{hobik}\\
	& = & \frac{4}{3} \sum_{i, k} \left| \left( J \mathring{h} (e_i, e_k)
	\right)^N \right|^2 \nonumber\\
	& = & \frac{4}{3} \sum_{i, k} \left( \left| J \mathring{h} (e_i, e_k)
	\right|^2 - \left| P \mathring{h} (e_i, e_k) \right|^2 \right) \nonumber\\
	& = & \frac{4}{3} \left( \normho^2 - \left| P \mathring{h} \right|^2
	\right) . \nonumber
	\end{eqnarray}

	For fixed $\alpha, \beta$, we choose an orthonormal frame $\{ e_i \}$, such
	that the $n \times n$ matrix $(J_{i \alpha} J_{j \beta} - J_{j \alpha} J_{i
		\beta})$ takes the form of (\ref{antisym}). Thus we get
	\begin{eqnarray*}
		&  & \sum_{i, j, k} \mathring{h}^{\alpha}_{i k} \mathring{h}^{\beta}_{j
			k} (J_{i \alpha} J_{j \beta} - J_{j \alpha} J_{i \beta})\\
		& = & \sum_{i, k} \mathring{h}^{\alpha}_{i k}
		\mathring{h}^{\beta}_{\tilde{i} k} (J_{i \alpha} J_{\tilde{i} \beta} -
		J_{\tilde{i} \alpha} J_{i \beta})\\
		& \leqslant & \sum_{i, k} \left| \mathring{h}^{\alpha}_{i k}
		\mathring{h}^{\beta}_{\tilde{i} k} \right| \sqrt{(J_{i \alpha})^2 +
			(J_{\tilde{i} \alpha})^2} \sqrt{(J_{i \beta})^2 + (J_{\tilde{i}
				\beta})^2}\\
		& \leqslant & \sum_{i, k} \left| \mathring{h}^{\alpha}_{i k}
		\mathring{h}^{\beta}_{\tilde{i} k} \right| | P e_{\alpha} | | P e_{\beta}
		|\\
		& \leqslant & \sqrt{\sum_{i, k} \left( \mathring{h}^{\alpha}_{i k}
			\right)^2} \sqrt{\sum_{i, k} \left( \mathring{h}^{\beta}_{\tilde{i} k}
			\right)^2} | P e_{\alpha} | | P e_{\beta} |\\
		& = & \left| \mathring{h}^{\alpha} \right| \left| \mathring{h}^{\beta}
		\right| | P e_{\alpha} | | P e_{\beta} | .
	\end{eqnarray*}
	Then
	\begin{eqnarray*}
		&  & \sum_{\alpha, \beta, i, j, k} \mathring{h}^{\alpha}_{i k}
		\mathring{h}^{\beta}_{j k} (J_{i \alpha} J_{j \beta} - J_{j \alpha} J_{i
			\beta})\\
		& \leqslant & \sum_{\alpha \neq \beta} | P e_{\alpha} | | P e_{\beta} |
		\left| \mathring{h}^{\alpha} \right| \left| \mathring{h}^{\beta} \right|\\
		& = & 2 | P e_{n + 1} | \left| \mathring{h}^{n + 1} \right| \sum_{\alpha
			> n + 1} | P e_{\alpha} | \left| \mathring{h}^{\alpha} \right| +
		\sum_{\tmscript{\begin{array}{c}
					\alpha, \beta > n + 1\\
					\alpha \neq \beta
				\end{array}}} | P e_{\alpha} | | P e_{\beta} | \left|
				\mathring{h}^{\alpha} \right| \left| \mathring{h}^{\beta} \right|\\
				& \leqslant & 2 \sqrt{\rho_1} \sum_{\alpha > n + 1} | P e_{\alpha} |
				\left| \mathring{h}^{\alpha} \right| + \left( \sum_{\alpha > n + 1} | P
				e_{\alpha} | \left| \mathring{h}^{\alpha} \right| \right)^2 .
			\end{eqnarray*}
			Using the Cauchy inequality, we have
			\[ \left( \sum_{\alpha > n + 1} | P e_{\alpha} | \left|
			\mathring{h}^{\alpha} \right| \right)^2 \leqslant \left( \sum_{\alpha > n
				+ 1} | P e_{\alpha} |^2 \right) \left( \sum_{\alpha > n + 1} \left|
			\mathring{h}^{\alpha} \right|^2 \right) = \theta_2 \rho_2 . \]
			So
			\[ \sum_{\alpha, \beta, i, j, k} \mathring{h}^{\alpha}_{i k}
			\mathring{h}^{\beta}_{j k} (J_{i \alpha} J_{j \beta} - J_{j \alpha} J_{i
				\beta}) \leqslant 2 \sqrt{\theta_2 \rho_1 \rho_2} + \theta_2 \rho_2 . \]
			Thus we obtain $S_1 \leqslant \frac{3}{n} S_2 + 3 \normho^2 + 8
			\sqrt{\theta_2 \rho_1 \rho_2} + 4 \theta_2 \rho_2$.
		\end{proof}

		Let $F : M \times [0, T) \rightarrow \mathbb{C}\mathbb{P}^{\frac{n +
				q}{2}}$ be a mean curvature flow with $2 \leqslant q \leqslant n - 6$. Suppose
		that the initial submanifold $M_0$ satisfies the pinching condition $\normho^2
		< k | H |^2 + l$, where $k = \frac{1}{n (n - 1)}, l = 2 - \frac{3}{n}$. Since
		$M_0$ is compact, there exists a small positive number $\varepsilon$, such
		that $M_0$ satisfies $\normho^2 < (k | H |^2 + l) (1 - \varepsilon)$.
		
		Now we prove that the pinching condition is preserved.
		
		\begin{theorem}
			\label{pinch}If $M_0$ satisfies $\normho^2 < (k | H |^2 + l) (1 -
			\varepsilon)$, then this condition holds for all time $t \in [0, T)$.
		\end{theorem}
		
		\begin{proof}
			Let $k_{\varepsilon} = k (1 - \varepsilon), l_{\varepsilon} = l (1 -
			\varepsilon)$. We set $U = \normho^2 - k_{\varepsilon} | H |^2 -
			l_{\varepsilon}$. From the evolution equations we have
			\begin{eqnarray*}
				\frac{1}{2} \left( \dt - \Delta \right) U & = & k_{\varepsilon} | \nabla H
				|^2 - \left| \nabla \mathring{h} \right|^2 - n \left( \normho^2 +
				k_{\varepsilon} | H |^2 \right)\\
				&  & + R_1 - \left( k_{\varepsilon} + \tfrac{1}{n} \right) R_2 + S_1 - 3
				\left( k_{\varepsilon} + \tfrac{1}{n} \right) S_2 .
			\end{eqnarray*}
			Now we show that $\left( \dt - \Delta \right) U$ is negative at all points
			where $\normho^2 = k_{\varepsilon} | H |^2 + l_{\varepsilon}$.
			
			From Lemma \ref{nqineq}, if $\normho^2 = k_{\varepsilon} | H |^2 +
			l_{\varepsilon}$, we get the following estimates.
			\begin{eqnarray*}
				R_1 - \left( k_{\varepsilon} + \tfrac{1}{n} \right) R_2 & \leqslant &
				\left( \normho^2 - k_{\varepsilon} | H |^2 \right) | h |^2 - \left(
				\tfrac{1}{n} - k_{\varepsilon} \right) | H |^2 \rho_2 + 2 \rho_1 \rho_2 +
				\frac{1}{2} \rho_2^2\\
				& \leqslant & l_{\varepsilon} | h |^2 - (n - 2) k_{\varepsilon} | H |^2
				\rho_2 + 2 \rho_1 \rho_2 + \frac{1}{2} \rho_2^2\\
				& = & \left( k_{\varepsilon} + \tfrac{1}{n} \right) l_{\varepsilon} | H
				|^2 + l_{\varepsilon}^2 - (n - 2) (\rho_1 + \rho_2 - l_{\varepsilon})
				\rho_2 + 2 \rho_1 \rho_2 + \frac{1}{2} \rho_2^2\\
				& = & \left( k_{\varepsilon} + \tfrac{1}{n} \right) l_{\varepsilon} | H
				|^2 + l_{\varepsilon}^2 + (n - 2) l_{\varepsilon} \rho_2 - (n - 4) \rho_1
				\rho_2 - \left( n - \frac{5}{2} \right) \rho_2^2,
			\end{eqnarray*}
			and
			\[ S_1 - 3 \left( k_{\varepsilon} + \tfrac{1}{n} \right) S_2 \leqslant 3
			\normho^2 + 8 \sqrt{\theta_2 \rho_1 \rho_2} + 4 \theta_2 \rho_2 \leqslant
			3 k_{\varepsilon} | H |^2 + 3 l_{\varepsilon} + 8 \sqrt{\theta_2 \rho_1
				\rho_2} + 4 \theta_2 \rho_2 . \]
			Thus, at a point where $U = 0$, we get
			\begin{eqnarray}
			\frac{1}{2} \left( \dt - \Delta \right) U & \leqslant & k_{\varepsilon} |
			\nabla H |^2 - \left| \nabla \mathring{h} \right|^2 - n (2 k_{\varepsilon}
			| H |^2 + l_{\varepsilon}) \nonumber\\
			&  & + \left( k_{\varepsilon} + \tfrac{1}{n} \right) l_{\varepsilon} | H
			|^2 + l_{\varepsilon}^2 + (n - 2) l_{\varepsilon} \rho_2 - (n - 4) \rho_1
			\rho_2 - \left( n - \frac{5}{2} \right) \rho_2^2 \nonumber\\
			&  & + 3 k_{\varepsilon} | H |^2 + 3 l_{\varepsilon} + 8 \sqrt{\theta_2
				\rho_1 \rho_2} + 4 \theta_2 \rho_2  \label{paraU}\\
			& = & k_{\varepsilon} | \nabla H |^2 - \left| \nabla \mathring{h}
			\right|^2 + \left[ (3 - 2 n) k_{\varepsilon} + \left( k_{\varepsilon} +
			\tfrac{1}{n} \right) l_{\varepsilon} \right] | H |^2 \nonumber\\
			&  & + (3 - n + l_{\varepsilon}) l_{\varepsilon} + (n - 2)
			l_{\varepsilon} \rho_2 - \left( n - \frac{5}{2} \right) \rho_2^2
			\nonumber\\
			&  & - (n - 4) \rho_1 \rho_2 + 8 \sqrt{\theta_2 \rho_1 \rho_2} + 4
			\theta_2 \rho_2 . \nonumber
			\end{eqnarray}
			By the definition of $k_{\varepsilon}, l_{\varepsilon}$, we have
			\[ (3 - 2 n) k_{\varepsilon} + \left( k_{\varepsilon} + \tfrac{1}{n} \right)
			l_{\varepsilon} \leqslant 0. \]
			On the other hand, we have
			\begin{equation}
			- (n - 4) \rho_1 \rho_2 + 8 \sqrt{\theta_2 \rho_1 \rho_2} \leqslant
			\frac{16 \theta_2}{n - 4} \label{srP1P2}
			\end{equation}
			and
			\begin{equation}
			(3 - n + l_{\varepsilon}) l_{\varepsilon} + (n - 2) l_{\varepsilon} \rho_2
			\leqslant \frac{(n - 2)^2 l_{\varepsilon}}{4 (n - 3 - l_{\varepsilon})}
			\rho_2^2 \leqslant \frac{(n - 2)^2 l}{4 (n - 3 - l)} \rho_2^2 .
			\label{P22}
			\end{equation}
			Hence, the RHS of (\ref{paraU}) is not greater than
			\[ k_{\varepsilon} | \nabla H |^2 - \left| \nabla \mathring{h} \right|^2 +
			\left[ \frac{(n - 2)^2 l}{4 (n - 3 - l)} - \left( n - \frac{5}{2} \right)
			\right] \rho_2^2 + 4 \theta_2 \rho_2 + \frac{16 \theta_2}{n - 4} . \]
			Since $n \geqslant q + 6 \geqslant 8$, we have $\frac{(n - 2)^2 l}{4 (n - 3
				- l)} - \left( n - \frac{5}{2} \right) < 0$. Then we get
			\begin{eqnarray}
			&  & \left[ \frac{(n - 2)^2 l}{4 (n - 3 - l)} - \left( n - \frac{5}{2}
			\right) \right] \rho_2^2 + 4 \theta_2 \rho_2 + \frac{16 \theta_2}{n - 4}
			\nonumber\\
			& \leqslant & \frac{4}{n - \frac{5}{2} - \frac{(n - 2)^2 l}{4 (n - 3 -
					l)}} \theta_2^2 + \frac{16 \theta_2}{n - 4} \nonumber\\
			& \leqslant & \left[ \frac{4 (q - 1)}{n - \frac{5}{2} - \frac{(n - 2)^2
					l}{4 (n - 3 - l)}} + \frac{16}{n - 4} \right] \theta_2  \label{P222}\\
			& \leqslant & \left[ \frac{4 (n - 7)}{n - \frac{5}{2} - \frac{(n - 2)^2
					l}{4 (n - 3 - l)}} + \frac{16}{n - 4} \right] \theta_2 \nonumber\\
			& = & 8 \left[ 1 - \frac{- 60 + 76 n - 16 n^2 + n^3}{(n - 4) (- 18 + 42 n
				- 19 n^2 + 2 n^3)} \right] \theta_2 \leqslant 8 \theta_2 . \nonumber
			\end{eqnarray}
			By Lemma \ref{dA2}, we get
			\[ k_{\varepsilon} | \nabla H |^2 - \left| \nabla \mathring{h} \right|^2 + 8
			\theta_2 \leqslant 0. \]
			Therefore, the RHS of (\ref{paraU}) is nonpositive for small $\varepsilon$.
			Then the assertion follows from the maximum principle.
		\end{proof}
		
		Let $f_{\sigma} = \normho^2 / (k | H |^2 + l)^{1 - \sigma}$, where $\sigma \in
		(0, 1)$ is a positive constant. Then we have
		
		\begin{lemma}
			\label{ptf2}If $M_0$ satisfies $\normho^2 < (k | H |^2 + l) (1 -
			\varepsilon)$, then the following inequality holds along the mean curvature
			flow.
			\[ \dt f_{\sigma} \leqslant \Delta f_{\sigma} + \frac{2 k | \nabla
				f_{\sigma} | | \nabla | H |^2 |}{k | H |^2 + l} - \frac{2 f_{\sigma}}{3
				\normho^2} \left| \nabla \mathring{h} \right|^2 + 2 \sigma f_{\sigma} (|
			h |^2 + n) - 2 \varepsilon l f_{\sigma} . \]
		\end{lemma}
		
		\begin{proof}
			By a direct calculation, we have
			\begin{eqnarray*}
				\left( \dt - \Delta \right) f_{\sigma} & = & f_{\sigma} \left[
				\frac{1}{\normho^2} \left( \dt - \Delta \right) \normho^2 - \frac{1 -
					\sigma}{| H |^2 + l / k} \left( \dt - \Delta \right) | H |^2 \right]\\
				&  & + 2 (1 - \sigma) \frac{\langle \nabla f_{\sigma}, \nabla | H |^2
					\rangle}{| H |^2 + l / k} - \sigma (1 - \sigma) f_{\sigma}  \frac{| \nabla
					| H |^2 |^2}{| H |^2 + l / k} .
			\end{eqnarray*}
			Now we estimate $(\partial_t - \Delta) \normho^2$ and $(\partial_t - \Delta)
			| H |^2$. From (\ref{srP1P2}), (\ref{P222}) and Lemma \ref{nqineq}, we get
			\begin{eqnarray}
			\left( \dt - \Delta \right) \normho^2 & \leqslant & - 2 \left| \nabla
			\mathring{h} \right|^2 + 2 \normho^2 (| h |^2 - n + 3) - \frac{2}{n}
			\rho_2 | H |^2 \nonumber\\
			&  & + 4 \rho_1 \rho_2 + \rho_2^2 + 16 \sqrt{\theta_2 \rho_1 \rho_2} + 8
			\theta_2 \rho_2 \nonumber\\
			& \leqslant & - 2 \left| \nabla \mathring{h} \right|^2 + 2 \normho^2 (| h
			|^2 - n + 3) - \frac{2}{n} \rho_2 | H |^2 \nonumber\\
			&  & + 2 (n - 2) \rho_1 \rho_2 + \rho_2^2 + \frac{32 \theta_2}{n - 4} + 8
			\theta_2 \rho_2  \label{parahohic2}\\
			& = & - 2 \left| \nabla \mathring{h} \right|^2 + 2 \normho^2 (| h |^2 - n
			+ 3 + (n - 2) \rho_2) - \frac{2}{n} \rho_2 | H |^2 \nonumber\\
			&  & + (5 - 2 n) \rho_2^2 + \frac{32 \theta_2}{n - 4} + 8 \theta_2 \rho_2
			\nonumber\\
			& \leqslant & - 2 \left| \nabla \mathring{h} \right|^2 + 2 \normho^2 (| h
			|^2 - n + 3 + (n - 2) \rho_2) - \frac{2}{n} \rho_2 | H |^2 \nonumber\\
			&  & + 16 \theta_2 - \frac{(n - 2)^2 l}{2 (n - 3 - l)} \rho_2^2 .
			\nonumber
			\end{eqnarray}
			We also have
			\begin{eqnarray*}
				&  & - \frac{1 - \sigma}{| H |^2 + l / k} \left( \dt - \Delta \right) | H
				|^2\\
				& \leqslant & \frac{2 | \nabla H |^2}{| H |^2 + l / k} - \frac{2 (1 -
					\sigma) | H |^2}{| H |^2 + l / k} ( | h |^2 - \rho_2 + n)\\
				& \leqslant & \frac{2 k | \nabla H |^2}{\normho^2} - \left( 2 - \frac{2
					l}{k | H |^2 + l} \right) ( | h |^2 - \rho_2 + n) + 2 \sigma (| h |^2 +
				n)\\
				& = & \frac{2 k | \nabla H |^2}{\normho^2} + \frac{2 l (| h |^2 - \rho_2
					+ n)}{k | H |^2 + l} - 2 (| h |^2 - \rho_2 + n) + 2 \sigma (| h |^2 + n)\\
				& \leqslant & \frac{2 k | \nabla H |^2}{\normho^2} + \frac{2 l \left[ (1
					- \varepsilon) (k | H |^2 + l) + \frac{1}{n} | H |^2 - \rho_2 + n
					\right]}{k | H |^2 + l}\\
				&  & - 2 (| h |^2 - \rho_2 + n) + 2 \sigma (| h |^2 + n)\\
				& \leqslant & \frac{2 k | \nabla H |^2}{\normho^2} + \frac{2 l \left(
					\frac{1}{n} | H |^2 - \rho_2 + n \right)}{k | H |^2 + l} + 2 l (1 -
				\varepsilon) - 2 (| h |^2 - \rho_2 + n) + 2 \sigma (| h |^2 + n) .
			\end{eqnarray*}
			Then we get
			\begin{eqnarray*}
				&  & \left( \dt - \Delta \right) f_{\sigma}\\
				& \leqslant & f_{\sigma} \left[ \frac{1}{\normho^2} \left( \dt - \Delta
				\right) \normho^2 - \frac{1 - \sigma}{| H |^2 + l / k} \left( \dt - \Delta
				\right) | H |^2 \right] + \frac{2 | \nabla f_{\sigma} | | \nabla | H |^2
					|}{| H |^2 + l / k}\\
				& \leqslant & \frac{2 f_{\sigma}}{\normho^2} \left[ - \left| \nabla
				\mathring{h} \right|^2 + 8 \theta_2 + k | \nabla H |^2 \right] + \frac{2 |
					\nabla f_{\sigma} | | \nabla | H |^2 |}{| H |^2 + l / k}\\
				&  & + 2 f_{\sigma} [- 2 n + 3 + (n - 1) \rho_2 + l (1 - \varepsilon) +
				\sigma (| h |^2 + n)]\\
				&  & + 2 f_{\sigma} \left[ - \frac{1}{\normho} \left[ \tfrac{1}{n} \rho_2
				| H |^2 - \frac{(n - 2)^2 l}{4 (n - 3 - l)} \rho_2^2 \right] + \frac{l
					\left( \frac{1}{n} | H |^2 - \rho_2 + n \right)}{k | H |^2 + l} \right]\\
				& \leqslant & \frac{2 f_{\sigma}}{\normho^2} \left[ - \left| \nabla
				\mathring{h} \right|^2 + 8 \theta_2 + k | \nabla H |^2 \right] + \frac{2 |
					\nabla f_{\sigma} | | \nabla | H |^2 |}{| H |^2 + l / k}\\
				&  & + 2 f_{\sigma} [(n - 1) (\rho_2 - l) + \sigma (| h |^2 + n) -
				\varepsilon l]\\
				&  & + \frac{2 f_{\sigma}}{k | H |^2 + l} \left[ - \tfrac{1}{n} \rho_2 |
				H |^2 - \frac{(n - 2)^2 l}{4 (n - 3 - l)} \rho_2^2 + l \left( \tfrac{1}{n}
				| H |^2 - \rho_2 + n \right) \right] .
			\end{eqnarray*}
			By Lemma \ref{dA2} and the condition $n \geqslant q + 6 \geqslant 8$, we have
			\[ - \left| \nabla \mathring{h} \right|^2 + 8 \theta_2 + k | \nabla H |^2
			\leqslant - \frac{1}{3} \left| \nabla \mathring{h} \right|^2 . \]
			Using (\ref{P22}), we get
			\begin{eqnarray*}
				&  & - \tfrac{1}{n} \rho_2 | H |^2 - \frac{(n - 2)^2 l}{4 (n - 3 - l)}
				\rho_2^2 + l \left( \tfrac{1}{n} | H |^2 - \rho_2 + n \right)\\
				& \leqslant & - \tfrac{1}{n} \rho_2 | H |^2 - (3 - n + l) l - (n - 2) l
				\rho_2 + l \left( \tfrac{1}{n} | H |^2 - \rho_2 + n \right)\\
				& = & (n - 1) (l - \rho_2) (k | H |^2 + l) .
			\end{eqnarray*}
			Therefore, we complete the proof of this lemma.
		\end{proof}
		
		\subsection{The case of $q \geqslant n - 4$}\
		
		We choose an orthonormal frame $\{ e_{\alpha} \}$ for the normal space such
		that $H = | H | e_{n + 1}$. Let $\rho_1 = \left| \mathring{h}^{n + 1}
		\right|^2$, $\rho_2 = \sum_{\alpha > n + 1} \left| \mathring{h}^{\alpha}
		\right|^2$. We have
		
		\begin{lemma}
			\label{nqineq2}{\tmdummy}
			
			\begin{enumerateroman}
				\item  $R_1 \leqslant | h |^4 - \frac{2}{n} \rho_2 | H |^2 + 2 \normho^2
				\rho_2 - \frac{3}{2} \rho_2^2$,
				
				\item $R_2 = | H |^2 (| h |^2 - \rho_2)$,
				
				\item $S_1 \leqslant \frac{3}{n} S_2 + (2 n + 3) \normho^2$.
			\end{enumerateroman}
		\end{lemma}
		
		\begin{proof}
			Using the same proof as Lemma \ref{nqineq}, we obtain (i) and (ii).
			
			Now we re-estimate $S_1$. From (\ref{s2nP}) and (\ref{hobik}), we have
			\[ S_1 \leqslant \frac{3}{n} S_2 + 3 \normho^2 + 4 \sum_{\alpha, \beta, i,
				j, k} \mathring{h}^{\alpha}_{i k} \mathring{h}^{\beta}_{j k} (J_{i
				\alpha} J_{j \beta} - J_{i \beta} J_{j \alpha}) . \]
			With a local orthonormal frame, let $v$ be a vector given by $v_k =
			\mathring{h}^{\alpha}_{i k} J_{i \alpha}$. We define two tensors $D$ and $E$
			by
			\[ D_{i j k} = \mathring{h}^{\alpha}_{i j} J_{k \alpha} +
			\frac{\mathring{h}^{\alpha}_{i k} J_{j \alpha} + \mathring{h}^{\alpha}_{j
					k} J_{i \alpha}}{n + \eta}, \qquad E_{i j k} = - \frac{2 \delta_{i j}
				v_k}{(n + \eta) \eta} + \frac{\delta_{i k} v_j + \delta_{j k} v_i}{\eta},
			\]
			where $\eta = \sqrt{n^2 + n - 2} .$ Then we have $\langle D - E, E \rangle =
			0$. This implies $| D |^2 \geqslant | E |^2$. By the definitions of $D$ and
			$E$, we get
			\[ | D |^2 = \frac{5 n - 2 - 4 \eta}{(n - 2)^2} \left( n
			\mathring{h}^{\alpha}_{i j} \mathring{h}^{\beta}_{i j} J_{k \alpha} J_{k
				\beta} + 2 \mathring{h}^{\alpha}_{i k} \mathring{h}^{\beta}_{j k} J_{i
				\beta} J_{j \alpha} \right), \]
			\[ | E |^2 = \frac{2 (5 n - 2 - 4 \eta)}{(n - 2)^2} \mathring{h}^{\alpha}_{i
				k} \mathring{h}^{\beta}_{j k} J_{i \alpha} J_{j \beta} . \]
			Thus we obtain
			\[ \sum_{\alpha, \beta, i, j, k} \mathring{h}^{\alpha}_{i k}
			\mathring{h}^{\beta}_{j k} (J_{i \alpha} J_{j \beta} - J_{i \beta} J_{j
				\alpha}) \leqslant \frac{n}{2} \sum_{i, j, k} \left( \sum_{\alpha}
			\mathring{h}^{\alpha}_{i j} J_{k \alpha} \right)^2 \leqslant \frac{n}{2}
			\normho^2 . \]
			
		\end{proof}
		
		Define a function $\psi : [0, + \infty) \rightarrow \mathbb{R}$ by
		\begin{equation}\label{func2}
		\psi (x) := \tfrac{9}{n^2 - 3 n - 3} + \tfrac{n (n - 3)}{n^3 - 4 n^2 +
			3} x - \tfrac{3 \sqrt{x^2 + \frac{2}{n} (n - 1) (n^2 - 3) x + 9 (n -
				1)^2}}{n^3 - 4 n^2 + 3} .
		\end{equation}
		Let $\mathring{\psi} (x) = \psi (x) - \frac{x}{n}$. From Lemma \ref{psi} in
		the Appendix, the function $\mathring{\psi}$ has the following properties.
		
		\begin{lemma}
			\label{psio}The function $\mathring{\psi}$ satisfies
			\begin{enumerateroman}
				\item $0 \leqslant \mathring{\psi}' (x) < \frac{1}{n (n - 1)}$, $0
				\leqslant \mathring{\psi} (x) \leqslant \frac{x}{n (n - 1)}$,
				
				\item $\max_{x \geqslant 0} \left( 2 x \mathring{\psi}'' (x) +
				\mathring{\psi}' (x) \right) < \frac{2 (n - 4)}{n (n + 8)}$,
				
				\item $3 \mathring{\psi} (x) + \left( \mathring{\psi} (x) - x
				\mathring{\psi}' (x) \right) \left( \mathring{\psi} (x) + \frac{x}{n} + n
				\right) \leqslant 0$, and the equalities hold if and only $x = 0$,
				
				\item $0 \leqslant x \mathring{\psi}' (x) - \mathring{\psi} (x) < 2$.
			\end{enumerateroman}
		\end{lemma}
		
		For convenience, we denote $\mathring{\psi} (| H |^2)$, $\mathring{\psi}' (| H
		|^2)$, $\mathring{\psi}'' (| H |^2)$ by $\mathring{\psi}$, $\mathring{\psi}'$,
		$\mathring{\psi}''$, respectively. Let $F : M^n \times [0, T) \rightarrow
		\mathbb{C}\mathbb{P}^{\frac{n + q}{2}}$ $(q \geqslant n - 4 \geqslant 2)$ be
		a mean curvature flow whose initial value $M_0$ satisfies $| h |^2 < \psi(|H|^2)$.
		Since $M_0$ is compact, there exists a small positive number $\varepsilon$,
		such that $M_0$ satisfies $\normho^2 < \mathring{\psi} - \varepsilon | H |^2 -
		\varepsilon$.
		
		\begin{theorem}
			\label{prehico}If the initial value $M_0$ satisfies $\normho^2 <
			\mathring{\psi} - \varepsilon | H |^2 - \varepsilon$, then this pinching
			condition holds for all $t \in [0, T)$.
		\end{theorem}
		
		\begin{proof}
			From Lemma \ref{dA2} and Lemma \ref{nqineq2}, we get
			\begin{eqnarray}
			\left( \dt - \Delta \right) \normho^2 & = & - 2 \left| \nabla \mathring{h}
			\right|^2 - 2 n \normho^2 + 2 R_1 - \frac{2}{n} R_2 + 2 S_1 - \frac{6}{n}
			S_2 \nonumber\\
			& \leqslant & - 2 \left| \nabla \mathring{h} \right|^2 + 2 \normho^2 | h
			|^2 - \frac{2}{n} \rho_2 | H |^2 + 4 \normho^2 \rho_2 + 2 (n + 3)
			\normho^2 .  \label{parahoh2}
			\end{eqnarray}
			We have the following evolution equation of $\mathring{\psi}$.
			\begin{eqnarray}
			\left( \dt - \Delta \right) \mathring{\psi} & = & - 2 \mathring{\psi}'
			\cdot | \nabla H |^2 - \mathring{\psi}'' \cdot | \nabla | H |^2 |^2 + 2
			\mathring{\psi}' (n | H |^2 + R_2 + 3 S_2) \nonumber\\
			& \geqslant & - 2 \left( \mathring{\psi}' + 2 \mathring{\psi}'' \cdot | H
			|^2 \right) | \nabla H |^2 + 2 \mathring{\psi}' \cdot | H |^2 (| h |^2 + n
			- \rho_2) .  \label{parafh2}
			\end{eqnarray}
			Let $U = \normho^2 - \mathring{\psi} + \varepsilon | H |^2 + \varepsilon$.
			We obtain
			\begin{eqnarray*}
				\frac{1}{2} \left( \dt - \Delta \right) U & \leqslant & \left(
				\mathring{\psi}' - \varepsilon + 2 \mathring{\psi}'' \cdot | H |^2 \right)
				| \nabla H |^2 - \left| \nabla \mathring{h} \right|^2\\
				&  & + \normho^2 \left( \normho^2 + \tfrac{1}{n} | H |^2 + n + 3 \right)
				- \left( \mathring{\psi}' - \varepsilon \right) | H |^2 \left( \normho^2 +
				\tfrac{1}{n} | H |^2 + n \right)\\
				&  & + \rho_2 \left( - \tfrac{1}{n} | H |^2 + 2 \normho^2 + \left(
				\mathring{\psi}' - \varepsilon \right) | H |^2 \right) .
			\end{eqnarray*}
			By Lemma \ref{dA2} and Lemma \ref{psio} (ii), the first line of the RHS of
			the formula above is nonpositive. From $\normho^2 = U + \mathring{\psi} -
			\varepsilon | H |^2 - \varepsilon$, we obtain
			\begin{eqnarray}
			\frac{1}{2} \left( \dt - \Delta \right) U & \leqslant & U \left( U + 2
			\mathring{\psi} + n + 3 - 2 \varepsilon + \left( \tfrac{1}{n} -
			\mathring{\psi}' - \varepsilon \right) | H |^2 + 2 \rho_2 \right)
			\nonumber\\
			&  & + \mathring{\psi} \left( \mathring{\psi} + \tfrac{1}{n} | H |^2 + n
			+ 3 \right) - \mathring{\psi}' \cdot | H |^2 \left( \mathring{\psi} +
			\tfrac{1}{n} | H |^2 + n \right) \nonumber\\
			&  & + \varepsilon \left[ - n - 3 + \varepsilon + \mathring{\psi}' \cdot
			| H |^2 - 2 \mathring{\psi} \right.  \label{paraUhico}\\
			&  & \quad \left. + | H |^2 \left( \mathring{\psi}' \cdot | H |^2 -
			\mathring{\psi} - 3 - \tfrac{1}{n} + \varepsilon \right) \right]
			\nonumber\\
			&  & + \rho_2 \left( - \tfrac{1}{n} | H |^2 + 2 \mathring{\psi} +
			\mathring{\psi}' | H |^2 - \varepsilon (3 | H |^2 + 2) \right) . \nonumber
			\end{eqnarray}
			This together with Lemma \ref{psio} implies
			\begin{equation}
			\frac{1}{2} \left( \dt - \Delta \right) U < U \left( U + 2 \mathring{\psi}
			+ n + 3 - 2 \varepsilon + \left( \tfrac{1}{n} - \mathring{\psi}' -
			\varepsilon \right) | H |^2 \right) . \label{paraUleU}
			\end{equation}
			Then the assertion follows from the maximum principle.
		\end{proof}
		
		Now we prove that the mean curvature flow has finite maximal existence time in
		this case.
		
		\begin{lemma}
			If the initial value $M_0$ satisfies $\normho^2 < \mathring{\psi}$, then $T$
			is finite.
		\end{lemma}
		
		\begin{proof}
			Let $U = \normho^2 - \mathring{\psi}$. Then $U < 0$ holds for all $t \in
			[0, T)$. From (\ref{paraUleU}), we have
			\begin{eqnarray*}
				\left( \dt - \Delta \right) U & \leqslant & 2 U \left( U + 2
				\mathring{\psi} + n + 3 + \left( \tfrac{1}{n} - \mathring{\psi}' \right) |
				H |^2 \right)\\
				& \leqslant & 2 U \left( U + 2 \mathring{\psi} \right)\\
				& = & 2 U \left( 2 \normho^2 - U \right) \leqslant - 2 U^2 .
			\end{eqnarray*}
			From the maximum principle, $U$ will blow up in finite time. Therefore, $T$
			must be finite.
		\end{proof}
		
		Let $f_{\sigma} = \normho^2 / \mathring{\psi}^{1 - \sigma}$, where $\sigma \in
		(0, 1)$ is a positive constant. Then we have
		
		\begin{lemma}
			\label{ptf3}If $M_0$ satisfies $\normho^2 < \mathring{\psi}$, then there
			exists a small positive constant $\varepsilon$, such that the following
			inequality holds along the mean curvature flow.
			\[ \dt f_{\sigma} \leqslant \Delta f_{\sigma} + \frac{2}{\mathring{\psi}} |
			\nabla f_{\sigma} | \left| \nabla \mathring{\varphi} \right| - \frac{2
				\varepsilon f_{\sigma}}{\normho^2} \left| \nabla \mathring{h} \right|^2 +
			2 \sigma | h |^2 f_{\sigma} + 4 n f_{\sigma} . \]
		\end{lemma}
		
		\begin{proof}
			By a straightforward calculation, we have
			\begin{eqnarray*}
				\left( \dt - \Delta \right) f_{\sigma} & = & f_{\sigma} \left[
				\frac{1}{\normho^2} \left( \dt - \Delta \right) \normho^2 - \frac{1 -
					\sigma}{\mathring{\psi}} \left( \dt - \Delta \right) \mathring{\psi}
				\right]\\
				&  & + 2 (1 - \sigma) \frac{\left\langle \nabla f_{\sigma}, \nabla
					\mathring{\psi} \right\rangle}{\mathring{\psi}} - \sigma (1 - \sigma)
				f_{\sigma}  \frac{| \nabla \mathring{\psi} |^2}{|
					\mathring{\psi} |^2} .
			\end{eqnarray*}
			Using (\ref{parahoh2}) and (\ref{parafh2}), we have
			\begin{eqnarray*}
				\left( \dt - \Delta \right) f_{\sigma} & \leqslant & 2 f_{\sigma} \left[ -
				\frac{\left| \nabla \mathring{h} \right|^2}{\normho^2} + (1 - \sigma)
				\frac{\mathring{\psi}' + 2 H^2 \mathring{\psi}''}{\mathring{\psi}} |
				\nabla H |^2 \right]\\
				&  & + 2 f_{\sigma} \left[ | h |^2 + n + 3 - (1 - \sigma)
				\frac{\mathring{\psi}' \cdot | H |^2}{\mathring{\psi}} (| h |^2 + n)
				\right]\\
				&  & + 2 f_{\sigma} \rho_2 \left[ - \frac{| H |^2}{n \normho^2} + 2 + (1
				- \sigma) \frac{\mathring{\psi}' \cdot | H |^2}{\mathring{\psi}} \right] +
				\frac{2}{\mathring{\psi}} | \nabla f_{\sigma} | \left| \nabla
				\mathring{\psi} \right| .
			\end{eqnarray*}
			From Lemma \ref{dA2} and Lemma \ref{psio} (ii), we have
			\begin{eqnarray*}
				&  & - \frac{\left| \nabla \mathring{h} \right|^2}{\normho^2} + (1 -
				\sigma) \frac{\mathring{\psi}' + 2 H^2 \mathring{\psi}''}{\mathring{\psi}}
				| \nabla H |^2\\
				& \leqslant & - \frac{\left| \nabla \mathring{h} \right|^2}{\normho^2} +
				(1 - \sigma) \frac{(1 - \varepsilon) \frac{2 (n - 4)}{n (n +
						8)}}{\mathring{\psi}} | \nabla H |^2\\
				& \leqslant & - \frac{\left| \nabla \mathring{h} \right|^2}{\normho^2} +
				\frac{(1 - \varepsilon) \left| \nabla \mathring{h} \right|^2}{\normho^2} =
				- \frac{\varepsilon \left| \nabla \mathring{h} \right|^2}{\normho^2} .
			\end{eqnarray*}
			By Lemma \ref{psio} (iii), we get
			\begin{eqnarray*}
				&  & | h |^2 + n + 3 - (1 - \sigma) \frac{\mathring{\psi}' \cdot | H
					|^2}{\mathring{\psi}} (| h |^2 + n)\\
				& \leqslant & (1 - \sigma) \frac{\mathring{\psi} - \mathring{\psi}' \cdot
					| H |^2}{\mathring{\psi}} | h |^2 + n + 3 + \sigma | h |^2\\
				& \leqslant & 2 n + \sigma | h |^2 .
			\end{eqnarray*}
			By Lemma \ref{psio} (i), we have
			\[ - \frac{| H |^2}{n \normho^2} + 2 + (1 - \sigma) \frac{\mathring{\psi}'
				\cdot | H |^2}{\mathring{\psi}} \leqslant \frac{- \frac{1}{n} | H |^2 + 2
				\mathring{\psi} + \mathring{\psi}' \cdot | H |^2}{\mathring{\psi}}
			\leqslant 0. \]
			This completes the proof of the lemma.
		\end{proof}

		\section{An estimate for traceless second fundamental form}

		Let $F : M^n \times [0, T) \rightarrow \mathbb{C}\mathbb{P}^{\frac{n +
				q}{2}}$ be a mean curvature flow. Suppose that the initial value $M_0$
		satisfies the condition in Theorem \ref{theoh}. We put
		\[ W = \left\{ \begin{array}{ll}
		\mathring{\varphi}_{\varepsilon}, & q = 1,\\
		\frac{1}{n (n - 1)} | H |^2 + 2 - \frac{3}{n}, & 2 \leqslant q < n - 4,\\
		\mathring{\psi}, & q \geqslant n - 4.
		\end{array} \right. \]
		By the conclusions of the previous section, there exists a sufficiently small
		positive number $\varepsilon$, such that for all $t \in [0, T)$, the following
		pinching condition holds.
		\[ \normho^2 < W - \varepsilon | H |^2 - \varepsilon . \]
		From this inequality and the definition of $W$, we have $\varepsilon | H |^2 +
		\varepsilon < W < \frac{| H |^2}{n (n - 1)} + n$.
		
		We introduce an auxiliary function:
		\[ f_{\sigma} = \frac{\normho^2}{W^{1 - \sigma}}, \]
		In this section, we will show that $f_{\sigma}$ decays exponentially.
		
		\begin{lemma}
			\label{ptf}There exist positive constants $\varepsilon$ and $C_1$ depending
			on $M_0$, such that
			\[ \dt f_{\sigma} \leqslant \Delta f_{\sigma} + \frac{2 C_1}{\normho} |
			\nabla f_{\sigma} | \left| \nabla \mathring{h} \right| -
			\frac{\varepsilon f_{\sigma}}{\normho^2} \left| \nabla \mathring{h}
			\right|^2 + \frac{4 n \sigma}{\varepsilon} f_{\sigma} W + (\chi - 2
			\varepsilon) f_{\sigma}, \]
			where $\chi = 5 n \frac{\max \{ q - n + 5, 0 \}}{q - n + 5}$.
		\end{lemma}
		
		\begin{proof}
			Combining the conclusions of Lemmas \ref{ptf}, \ref{ptf2} and \ref{ptf3}, we
			have the following inequality with some suitable small $\varepsilon$.
			\[ \dt f_{\sigma} \leqslant \Delta f_{\sigma} + \frac{2}{W} | \nabla
			f_{\sigma} | | \nabla W | - \frac{\varepsilon f_{\sigma}}{\normho^2}
			\left| \nabla \mathring{h} \right|^2 + 2 \sigma f_{\sigma} (| h |^2 + n)
			+ (\chi - 2 \varepsilon) f_{\sigma} . \]
			By the definition of $W$, there exists a constant $B_1$ such that $| \nabla
			W | < B_1 | \nabla | H |^2 |$. Let $C_1$ be a constant such that $2 B_1 | H
			| / \sqrt{W} < \frac{1}{n} C_1$. From Lemma \ref{dA2}, we have
			\begin{equation}
			\frac{| \nabla W |}{W} \leqslant \frac{2 B_1 | H | | \nabla H |}{\sqrt{W}
				\normho} \leqslant \frac{\frac{1}{n} C_1 | \nabla H |}{\normho} \leqslant
			\frac{C_1 \left| \nabla \mathring{h} \right|}{\normho} .
			\label{aopaodfsdh2}
			\end{equation}
			We also have
			\[ | h |^2 + n < W + \frac{1}{n} | H |^2 + n < W + n (| H |^2 + 1) < \frac{2
				n}{\varepsilon} W. \]
			Thus we complete the proof of this lemma.
		\end{proof}
		
		We need the following estimate for the the Laplacian of $\normho^2$.
		
		\begin{lemma}
			\label{lapa0}There exists a positive constant $C_2 > 1$, such that
			\[ \Delta \normho^2 \geqslant 2 \left\langle \mathring{h}, \nabla^2 H
			\right\rangle + 2 \left( \varepsilon \normho^2 - C_2 \right) W. \]
		\end{lemma}
		
		\begin{proof}
			We have the following identity
			\[ \Delta \normho^2 = 2 \left\langle \mathring{h}, \nabla^2 H \right\rangle
			+ 2 \left| \nabla \mathring{h} \right|^2 - 2 R_1 + 2 R_3 + 2 n \normho^2
			- 2 S_1 + \tfrac{6}{n} S_2 + 6 S_3, \]
			where $R_3 = H^{\alpha} h^{\alpha}_{i k} h^{\beta}_{i j} h^{\beta}_{j k}$,
			$S_3 = \mathring{h}^{\alpha}_{i j} H^{\beta} J_{i \alpha} J_{j \beta}$. Then
			there exists a positive constant $B_2$ only depending on $n$, such that
			\[ \Delta \normho^2 \geqslant 2 \left\langle \mathring{h}, \nabla^2 H
			\right\rangle + 2 (R_3 - R_1) - 2 B_2 | h |^2 . \]

We choose a local orthonormal frame, such that $H = | H | e_{n + 1}$ and
$\mathring{h}^1 = \tmop{diag} \left( \mathring{\lambda}_1, \cdots,
\mathring{\lambda}_n \right)$. Let $\rho_1 = \left| \mathring{h}^{n + 1}
\right|^2$, $\rho_2 = \sum_{\alpha > n + 1} \left| \mathring{h}^{\alpha}
\right|^2$. Expanding $R_3$, we get
\begin{eqnarray*}
R_3 & = & \frac{1}{n^2} | H |^4 + \frac{3 \rho_1 + \rho_2}{n} | H |^2 + |
H | \sum_{\alpha, i} \mathring{\lambda}_i \left( \mathring{h}^{\alpha}_{i
i} \right)^2 + \frac{| H |}{2} \sum_{\tmscript{\begin{array}{c}
\alpha > 1\\
i \neq j
\end{array}}} \left( \mathring{\lambda}_i + \mathring{\lambda}_j \right)
\left( \mathring{h}^{\alpha}_{i j} \right)^2 .
\end{eqnarray*}
By Lemma 2.6 of {\cite{MR1289187}} or Proposition 1.6 of {\cite{MR1633163}},
we have
\begin{eqnarray*}
\sum_{\alpha, i} \mathring{\lambda}_i \left( \mathring{h}^{\alpha}_{i i}
\right)^2 & \geqslant & - \frac{n - 2}{\sqrt{n (n - 1)}} \sqrt{\rho_1} 
\Bigg( \normho^2 - \sum_{\tmscript{\begin{array}{c}
\alpha > 1\\
i \neq j
\end{array}}} \left( \mathring{h}^{\alpha}_{i j} \right)^2 \Bigg)\\
& \geqslant & - \frac{n - 2}{\sqrt{n (n - 1)}}  \Bigg( \frac{1}{2} \left(
P_1 + \normho^2 \right) \normho - \sqrt{\rho_1}
\sum_{\tmscript{\begin{array}{c}
\alpha > 1\\
i \neq j
\end{array}}} \left( \mathring{h}^{\alpha}_{i j} \right)^2 \Bigg)\\
& \geqslant & - \frac{n - 2}{\sqrt{n (n - 1)}} \Bigg( \normho^3 -
\frac{1}{2} \normho \rho_2 - \sqrt{\rho_1}
\sum_{\tmscript{\begin{array}{c}
\alpha > 1\\
i \neq j
\end{array}}} \left( \mathring{h}^{\alpha}_{i j} \right)^2 \Bigg) .
\end{eqnarray*}
We also have
\[ \sum_{\tmscript{\begin{array}{c}
\alpha > 1\\
i \neq j
\end{array}}} \left( \mathring{\lambda}_i + \mathring{\lambda}_j \right)
\left( \mathring{h}^{\alpha}_{i j} \right)^2 \geqslant
\sum_{\tmscript{\begin{array}{c}
\alpha > 1\\
i \neq j
\end{array}}} - \sqrt{2 \left( \mathring{\lambda}_i^2 +
\mathring{\lambda}_j^2 \right)} \left( \mathring{h}^{\alpha}_{i j}
\right)^2 \geqslant - \sqrt{2 \rho_1} \sum_{\tmscript{\begin{array}{c}
\alpha > 1\\
i \neq j
\end{array}}} \left( \mathring{h}^{\alpha}_{i j} \right)^2 . \]
Note that $ \tfrac{n - 2}{\sqrt{n (n - 1)}} > \frac{\sqrt{2}}{2}$ if $n
\geqslant 6$, and $\sum_{\alpha > 1, i \neq j} \left(
\mathring{h}^{\alpha}_{i j} \right)^2 = 0$ if $q = 1$. We get
\begin{equation}
R_3 \geqslant \frac{1}{n^2} | H |^4 + \frac{3 \rho_1 + \rho_2}{n} | H |^2
- \tfrac{n - 2}{\sqrt{n (n - 1)}} | H | \left( \normho^3 - \frac{1}{2}
\normho \rho_2 \right) . \label{R3grea}
\end{equation}
From (\ref{R3grea}) and Lemma \ref{nqineq} (i), we get
\begin{eqnarray*}
R_3 - R_1 & \geqslant & \normho^2 \left( - \normho^2 + \frac{1}{n} | H |^2
- \tfrac{n - 2}{\sqrt{n (n - 1)}} | H | \normho \right)\\
&  & + \rho_2 \left( - 2 \normho^2 + \tfrac{n - 2}{2 \sqrt{n (n - 1)}} |
H | \normho + \frac{3}{2} \rho_2 \right)\\
& \geqslant & \normho^2 \left[ \varepsilon (| H |^2 + 1) - W +
\frac{1}{n} | H |^2 - \tfrac{n - 2}{\sqrt{n (n - 1)}} | H | \sqrt{W}
\right]\\
&  & - \rho_2 \left[ \normho \left( 2 \sqrt{W} - \tfrac{n - 2}{2 \sqrt{n
(n - 1)}} | H | \right) - \frac{3}{2} \rho_2 \right]\\
& \geqslant & \normho^2 \left[ \varepsilon (| H |^2 + 1) + \tfrac{(n - 2)
| H |^2}{n (n - 1)} - n - \tfrac{n - 2}{\sqrt{n (n - 1)}} | H |
\sqrt{\tfrac{| H |^2}{n (n - 1)} + n} \right]\\
&  & - \rho_2 \left[ \normho \left( 2 \sqrt{\tfrac{| H |^2}{n (n - 1)} +
n} - \tfrac{n - 2}{2 \sqrt{n (n - 1)}} | H | \right) - \frac{3}{2} \rho_2
\right]\\
& \geqslant & \normho^2 [\varepsilon (| H |^2 + 1) - (n - 1) n] - \rho_2
\left[ \normho \cdot 2 \sqrt{n} - \frac{3}{2} \rho_2 \right]\\
& \geqslant & \normho^2  [\varepsilon (| H |^2 + 1) - (n - 1) n] -
\frac{2 n}{3} \normho^2 .
\end{eqnarray*}
By choosing suitable large $C_2$, we complete the proof of this lemma.
\end{proof}

From (\ref{aopaodfsdh2}) and Lemma \ref{lapa0}, we have
\begin{eqnarray*}
\Delta f_{\sigma} & = & f_{\sigma} \left( \frac{\Delta \normho^2}{\normho^2}
- (1 - \sigma) \frac{\Delta W}{W} \right) - 2 (1 - \sigma) \frac{\langle
\nabla f_{\sigma}, \nabla W \rangle}{W} + \sigma (1 - \sigma) f_{\sigma} 
\frac{| \nabla W |^2}{W^2}\\
& \geqslant & \frac{f_{\sigma} \Delta \normho^2}{\normho^2} - (1 - \sigma)
\frac{f_{\sigma} \Delta W}{W} - \frac{2 C_1 | \nabla f_{\sigma} | \left|
\nabla \mathring{h} \right|}{\normho}\\
& \geqslant & \frac{2 \left\langle \mathring{h}, \nabla^2 H
\right\rangle}{W^{1 - \sigma}} + 2 \varepsilon f_{\sigma} W - 2 C_2
W^{\sigma} - (1 - \sigma) \frac{f_{\sigma} \Delta W}{W} - \frac{2 C_1 |
\nabla f_{\sigma} | \left| \nabla \mathring{h} \right|}{\normho} .
\end{eqnarray*}
Multiplying both sides of the above inequality by $f_{\sigma}^{p - 1}$, we get
\[ 2 \varepsilon f^p_{\sigma} W \leqslant f_{\sigma}^{p - 1} \Delta
f_{\sigma} + (1 - \sigma) \frac{f_{\sigma}^p \Delta W}{W} - \frac{2
f_{\sigma}^{p - 1} \left\langle \mathring{h}, \nabla^2 H
\right\rangle}{W^{1 - \sigma}} + \frac{2 C_1 f_{\sigma}^{p - 1} | \nabla
f_{\sigma} | \left| \nabla \mathring{h} \right|}{\normho} + 2 C_2
f_{\sigma}^{p - 1} W^{\sigma} . \]
The Young's inequality yields
\begin{eqnarray*}
2 C_2 f_{\sigma}^{p - 1} W^{\sigma} & \leqslant & 2 C_2 W \left[ \frac{1}{p}
\left( \frac{2 C_2}{\varepsilon W^{(1 - \sigma)}} \right)^p + \frac{p -
1}{p} \left( \frac{\varepsilon}{2 C_2} \right)^{\frac{p}{p - 1}}
f_{\sigma}^p \right]\\
& \leqslant & 2 C_2 \left( \frac{1}{p} \left( \frac{2 C_2}{\varepsilon^{(2
- \sigma)}} \right)^p + \frac{\varepsilon}{2 C_2} f_{\sigma}^p W \right)\\
& \leqslant & C_3^p + \varepsilon f_{\sigma}^p W,
\end{eqnarray*}
where $C_3$ is a positive constant depending on $M_0$. Thus
\begin{equation}
\varepsilon f^p_{\sigma} W \leqslant f_{\sigma}^{p - 1} \Delta f_{\sigma} +
(1 - \sigma) \frac{f_{\sigma}^p \Delta W}{W} - \frac{2 f_{\sigma}^{p - 1}
\left\langle \mathring{h}, \nabla^2 H \right\rangle}{W^{1 - \sigma}} +
\frac{2 C_1 f_{\sigma}^{p - 1} | \nabla f_{\sigma} | \left| \nabla
\mathring{h} \right|}{\normho} + C_3^p . \label{eC2f}
\end{equation}
Then integrate both side of (\ref{eC2f}) over $M_t$. By the divergence
theorem, we have
\begin{equation}
\int_{M_t} f_{\sigma}^{p - 1} \Delta f_{\sigma} \mathd \mu_t = - (p - 1)
\int_{M_t} f_{\sigma}^{p - 2} | \nabla f_{\sigma} |^2 \mathd \mu_t .
\label{eC2f1}
\end{equation}
From (\ref{aopaodfsdh2}), we have
\begin{eqnarray}
\int_{M_t} \frac{f_{\sigma}^p}{W} \Delta W \mathd \mu_t & = & - \int_{M_t}
\left\langle \nabla \left( \frac{f_{\sigma}^p}{W} \right), \nabla W
\right\rangle \mathd \mu_t \nonumber\\
& = & \int_{M_t} \left( - \frac{p f_{\sigma}^{p - 1}}{W} \langle \nabla
f_{\sigma}, \nabla W \rangle + \frac{f_{\sigma}^p}{W^2} | \nabla W |^2
\right) \mathd \mu_t \\
& \leqslant & \int_{M_t} \left( \frac{C_1 p f_{\sigma}^{p - 1}}{\normho} |
\nabla f_{\sigma} | \left| \nabla \mathring{h} \right| + \frac{C_1^2
f_{\sigma}^p}{\normho^2} \left| \nabla \mathring{h} \right|^2 \right) \mathd
\mu_t . \nonumber
\end{eqnarray}
We also have
\begin{eqnarray}
&  & - \int_{M_t} \frac{f_{\sigma}^{p - 1} \left\langle \mathring{h},
\nabla^2 H \right\rangle}{W^{1 - \sigma}} \mathd \mu_t \nonumber\\
& = & \int_{M_t} \nabla_i \left( \frac{f_{\sigma}^{p - 1}}{W^{1 - \sigma}}
\mathring{h}_{i j}^{\alpha} \right) \nabla_j H^{\alpha} \mathd \mu_t 
\nonumber\\
& = & \!\!\!\!\int_{M_t}\!\! \left[ \frac{(p - 1) f_{\sigma}^{p - 2}}{W^{1 - \sigma}}
\mathring{h}^{\alpha}_{i j} \nabla_i f_{\sigma}\! - \!\frac{(1 - \sigma)
f_{\sigma}^{p - 1}}{W^{2 - \sigma}}  \mathring{h}^{\alpha}_{i j} \nabla_i W
\!+\! \frac{f_{\sigma}^{p - 1}}{W^{1 - \sigma}} \nabla_i
\mathring{h}^{\alpha}_{i j} \right]\!\! \nabla_j H^{\alpha} \mathd \mu_t
\nonumber\\
& \leqslant & \int_{M_t} \left[ \frac{(p - 1) f_{\sigma}^{p - 1}}{\normho}
| \nabla f_{\sigma} | + \frac{f_{\sigma}^{p - 1}}{W^{2 - \sigma}}  \normho |
\nabla W | + \frac{f_{\sigma}^{p - 1}}{W^{1 - \sigma}} n \left| \nabla
\mathring{h} \right| \right] | \nabla H | \mathd \mu_t \label{eC2f2}\\
& \leqslant & \int_{M_t} \left[ \frac{(p - 1) f_{\sigma}^{p - 1}}{\normho}
| \nabla f_{\sigma} | + \frac{C_1 f_{\sigma}^{p - 1}}{W^{1 - \sigma}} 
\normho \left| \nabla \mathring{h} \right| + \frac{f_{\sigma}^{p - 1}}{W^{1
- \sigma}} n \left| \nabla \mathring{h} \right| \right] n \left| \nabla
\mathring{h} \right| \mathd \mu_t \nonumber\\
& \leqslant & \int_{M_t} \left[ \frac{n (p - 1) f_{\sigma}^{p -
1}}{\normho} | \nabla f_{\sigma} | \left| \nabla \mathring{h} \right| +
\frac{(C_1 + n^2) f_{\sigma}^p}{\normho^2} \left| \nabla \mathring{h}
\right|^2 \right] \mathd \mu_t . \nonumber
\end{eqnarray}
Putting (\ref{eC2f})-(\ref{eC2f2}) together, we get
\begin{equation}
\int_{M_t} W f_{\sigma}^p \mathd \mu_t \leqslant \frac{C_4}{\varepsilon}
\int_{M_t} \left[ \frac{p f_{\sigma}^{p - 1}}{\normho}  | \nabla f_{\sigma}
| \left| \nabla \mathring{h} \right| + \frac{f_{\sigma}^p}{\normho^2} 
\left| \nabla \mathring{h} \right|^2 + C_3^p \right] \mathd \mu_t,
\label{H2fsp}
\end{equation}
where $C_4$ is a positive constant independent of $t$.

Combining Lemma \ref{ptf} and (\ref{H2fsp}), we obtain
\begin{eqnarray}
\frac{\mathd}{\mathd t} \int_{M_t} f_{\sigma}^p \mathd \mu_t & = & p
\int_{M_t} f_{\sigma}^{p - 1} \frac{\partial f_{\sigma}}{\partial t} \mathd
\mu_t - \int_{M_t} f_{\sigma}^p | H |^2 \mathd \mu_t \nonumber\\
& \leqslant & p \int_{M_t} f_{\sigma}^{p - 2} \left[ - (p - 1) | \nabla
f_{\sigma} |^2 + \left( 2 C_1 + \frac{4 n \sigma C_4 p}{\varepsilon^2}
\right) \frac{f_{\sigma}}{\normho} | \nabla f_{\sigma} | \left| \nabla
\mathring{h} \right| \right. \nonumber\\
&  & \left. - \left( \varepsilon - \frac{4 n \sigma C_4}{\varepsilon^2}
\right) \frac{f_{\sigma}^2}{\normho^2} \left| \nabla \mathring{h} \right|^2
\right] \mathd \mu_t  \label{dtint}\\
&  & + p \int_{M_t} f_{\sigma}^p \left( \frac{4 n \sigma C_4
C_3^p}{\varepsilon^2} + \chi - 2 \varepsilon \right) \mathd \mu_t .
\nonumber
\end{eqnarray}

Now we show that the $L^p$-norm of $f_{\sigma}$ decays exponentially.

\begin{lemma}
\label{pnorm}There exist positive constants $C_5, p_0, \sigma_0$ depending
on $M_0$, such that for all $p \geqslant p_0$ and $\sigma \leqslant \sigma_0
/ \sqrt{p}$, we have
\[ \left( \int_{M_t} f_{\sigma}^p \mathd \mu_t \right)^{\frac{1}{p}} < C_5
\mathe^{- \varepsilon t} . \]
\end{lemma}

\begin{proof}
The expression in the square bracket of the right hand side of (\ref{dtint})
is a quadratic polynomial. With $p_0$ large enough and $\sigma_0$ small
enough, its discriminant satisfies $\left( 2 C_1 + \frac{4 n \sigma C_4
p}{\varepsilon^2} \right)^2 - 4 (p - 1) \left( \varepsilon - \frac{4 n
\sigma C_4}{\varepsilon^2} \right) < 0$. We also have $\frac{4 n \sigma C_4
C_3^p}{\varepsilon^2} < \varepsilon$. Then we get
\[ \frac{\mathd}{\mathd t} \int_{M_t} f_{\sigma}^p \mathd \mu_t \leqslant p
(\chi - \varepsilon) \int_{M_t} f_{\sigma}^p \mathd \mu_t . \]

Since $\chi = 5 n \frac{\max \{ q - n + 5, 0 \}}{q - n + 5}$, we have
$\int_{M_t} f_{\sigma}^p \mathd \mu_t \leqslant \mathe^{- p \varepsilon t}
\int_{M_0} f_{\sigma}^p \mathd \mu_0$ for $q < n - 4$, and $\int_{M_t}
f_{\sigma}^p \mathd \mu_t \leqslant \mathe^{5 n p T} \int_{M_0} f_{\sigma}^p
\mathd \mu_0$ for $q \geqslant n - 4$. Noting that $T$ is finite if $q
\geqslant n - 4$, we obtain the conclusion.
\end{proof}

Let $g_{\sigma} = f_{\sigma} \mathe^{\varepsilon t / 2}$. By the Sobolev
inequality on submanifolds {\cite{MR0365424}} and a Stampacchia iteration
procedure, we obtain that $g_{\sigma}$ is uniformly bounded for all $t$ (see
{\cite{MR772132}} or {\cite{lei2014sharp}} for the details). Then we obtain
the following theorem.

\begin{theorem}
\label{sa0h2}There exist positive constants $ \sigma $ and $C_0$
depending only on $M_0$, such that for all $t \in [0, T)$ we have
\[ \normho^2 \leqslant C_0  (| H |^2 + 1)^{1 - \sigma} \mathe^{- \varepsilon
t / 2} . \]
\end{theorem}

\

\section{A gradient estimate}

In the following, we derive an estimate for $| \nabla H |^2$ along the mean
curvature flow.

Firstly, the same as Proposition 4.3 in {\cite{PiSi2015}}, we have the
following result.

\begin{lemma}
\label{dtdH2}There exists a constant $C_6 > 1$ depending only on $n$, such
that
\[ \dt | \nabla H |^2 \leqslant \Delta | \nabla H |^2 + C_6 (| H |^2 + 1) |
\nabla h |^2 . \]
\end{lemma}

Secondly, we need the following estimates.

\begin{lemma}
\label{threedtlap}Along the mean curvature flow, we have
\begin{enumerateroman}
\item $\dt | H |^4 \geqslant \Delta | H |^4 - 24 n | H |^2 | \nabla h |^2
+ \frac{4}{n} | H |^6$,

\item $\dt \normho^2 \leqslant \Delta \normho^2 - \frac{1}{3} | \nabla h
|^2 + C_7 (| H |^2 + 1) \normho^2$,

\item $\dt \left( | H |^2 \normho^2 \right) \leqslant \Delta \left( | H
|^2 \normho^2 \right) - \frac{1}{6} | H |^2 | \nabla h |^2 + C_9 | \nabla
h |^2 + C_8 (| H |^2 + 1)^2 \normho^2$,
\end{enumerateroman}
where $C_7, C_8, C_9$ are sufficiently large constants.
\end{lemma}

\begin{proof}
(i) From Lemma \ref{evoinCP} (ii) we derive that
\[ \dt | H |^4 = \Delta | H |^4 - 4 | H |^2 | \nabla H |^2 - 2 | \nabla | H
|^2 |^2 + 4 | H |^2 (R_2 + 3 S_2) + 4 n | H |^4 . \]
From Lemma \ref{nqineq2}, we get $R_2 + 3 S_2 \geqslant R_2 \geqslant
\frac{1}{n} | H |^4$. Then from Lemma \ref{dA2}, we have $4 | H |^2 | \nabla
H |^2 + 2 | \nabla | H |^2 |^2 \leqslant 24 n | H |^2 | \nabla h |^2$.

(ii) We have
\[ \dt \normho^2 = \Delta \normho^2 - 2 \left| \nabla \mathring{h} \right|^2
- 2 n \normho^2 + 2 R_1 - \frac{2}{n} R_2 + 2 S_1 - \frac{6}{n} S_2 . \]
From Lemma \ref{dA2}, we get $2 \left| \nabla \mathring{h} \right|^2
\geqslant \frac{1}{3} | \nabla h |^2$. Choosing a large constant $C_7$, we
obtain inequality (ii).

(iii) It follows from the evolution equations that
\begin{eqnarray*}
\dt \left( | H |^2 \normho^2 \right) & = & \Delta \left( | H |^2 \normho^2
\right) + 2 | H |^2 \left( R_1 - \frac{R_2}{n} + S_1 - \frac{3 S_2}{n}
\right) + 2 \normho^2 (R_2 + 3 S_2)\\
&  & - 2 | H |^2 \left| \nabla \mathring{h} \right|^2 - 2 \normho^2 |
\nabla H |^2 - 2 \left\langle \nabla | H |^2, \nabla \normho^2
\right\rangle .
\end{eqnarray*}
From Lemma \ref{dA2}, we get $- 2 | H |^2 \left| \nabla \mathring{h}
\right|^2 \leqslant - \frac{1}{3} | H |^2 | \nabla h |^2$.

From the preserved pinching condition $\normho^2 < W$, we have
\[ 2 | H |^2 \left( R_1 - \frac{R_2}{n} + S_1 - \frac{3 S_2}{n} \right) + 2
\normho^2 (R_2 + 3 S_2) < C_8 (| H |^2 + 1)^2 \normho^2 . \]

Using Theorem \ref{sa0h2}, we have
\[ - 2 \left\langle \nabla | H |^2, \nabla \normho^2 \right\rangle \leqslant
8 | H |  | \nabla H |  \normho  | \nabla h | \leqslant 8 n \sqrt{C_0} | H
|  (| H |^2 + 1)^{\frac{1 - \sigma}{2}} | \nabla h |^2 . \]
By Young's inequality, there exists a positive constant $C_9$, such that $-
2 \left\langle \nabla | H |^2, \nabla \normho^2 \right\rangle \leqslant
\left( C_9 + \frac{1}{6}  | H |^2 \right) | \nabla h |^2$.
\end{proof}

Now we prove a gradient estimate for mean curvature.

\begin{theorem}
\label{dH2}For all $\eta \in \left( 0, \frac{\varepsilon^{1 / 2}}{8 n
\pi} \right)$, there exists a number $\Psi (\eta)$ depending on $\eta$
and $M_0$, such that
\[ | \nabla H |^2 < [(\eta | H |)^4 + \Psi (\eta)^2] \mathe^{- \varepsilon t
/ 4} . \]
\end{theorem}

\begin{proof}
Define a scalar
\[ f = \left( | \nabla H |^2 + B_1  \normho^2 + B_2  | H |^2 \normho^2
\right) \mathe^{\varepsilon t / 4} - (\eta | H |)^4, \]
where $B_1, B_2$ are two positive constants.

From Lemmas \ref{dtdH2} and \ref{threedtlap}, we obtain
\begin{eqnarray*}
\left( \dt - \Delta \right) f & \leqslant & \left[ \mathe^{\varepsilon t /
4} \left( C_6 - \frac{B_2}{6} \right) + 24 n \eta^4 \right] | H |^2 |
\nabla h |^2\\
&  & + \mathe^{\varepsilon t / 4} \left[ \left( C_6 - \frac{B_1}{3} + B_2
C_9 \right) | \nabla h |^2 + \sigma | \nabla H |^2 \right]\\
&  & + \mathe^{\varepsilon t / 4} [B_1 C_7 (| H |^2 + 1) + B_2 C_8 (| H
|^2 + 1)^2 + \sigma (B_1 + B_2  | H |^2)] \normho^2\\
&  & - \frac{4 \eta^4}{n} | H |^6 .
\end{eqnarray*}
We choose the constants $B_1$ and $B_2$, such that $C_6 - \frac{B_2}{6} < -
1$ and $C_6 - \frac{B_1}{3} + B_2 C_9 < - 1$. Then applying Theorem
\ref{sa0h2}, we get
\begin{eqnarray}
\left( \dt - \Delta \right) f & \leqslant & \mathe^{- \varepsilon t / 4}
\left[ B_3 (| H |^2 + 1)^2 (| H |^2 + 1)^{1 - \sigma} - \frac{4 \eta^4}{n}
| H |^6 \right] .  \label{Paraf}
\end{eqnarray}
Consider the expression in the bracket of (\ref{Paraf}). Since the coefficient
of $| H |^6$ is negative, it has a upper bound $\Psi_2 (\eta)$. Then we have
$\left( \dt - \Delta \right) f \leqslant \Psi_2 (\eta) \mathe^{- \varepsilon
t / 4}$. It follows from the maximum principle that $f$ is bounded. This
completes the proof of Theorem \ref{dH2}.

\end{proof}

\section{Convergence}

In order to estimate the diameter of $M_t$, we need the well-known Myers
theorem.

\begin{theorem}
{\textbf{(Myers)}} Let $\Gamma$ be a geodesic of length $l$ on $M$. If the
Ricci curvature satisfies $\tmop{Ric} (X) \geqslant (n - 1)
\frac{\pi^2}{l^2}$, for each unit vector $X \in T_x M$, at any point $x \in
\Gamma$, then $\Gamma$ has conjugate points.
\end{theorem}

Then we obtain the following lemma.

\begin{lemma}
\label{ric}Let $M$ be an n-dimensional submanifold in
$\mathbb{C}\mathbb{P}^{\frac{n + q}{2}}$. If there exist positive
constants $L$ and $\varepsilon$, $\varepsilon < \frac{1}{n (n - 1)}$, such
that $M$ satisfies $| h |^2 < \left( \frac{1}{n - 1} - \varepsilon \right) |
H |^2 + L$, $\max_M | H |^2 > 2 n L / \varepsilon$ and $| \nabla H | < 2
\eta^2 \max_M | H |^2$ for all $0 < \eta < \frac{\varepsilon^{1 / 2}}{8 n
\pi}$, then we have
\[ \frac{\min_M | H |^2}{\max_M | H |^2} > 1 - \eta \quad \tmop{and} \quad
\tmop{diam} M \leqslant (2 \eta | H |_{\max})^{- 1} . \]
\end{lemma}

\begin{proof}
By Proposition 2 in {\cite{MR1458750}}, the Ricci curvature of $M$ satisfies
\begin{eqnarray*}
\tmop{Ric} & \geqslant & \tfrac{n - 1}{n} \left( n + \tfrac{2}{n} | H |^2
- | h |^2 - \tfrac{n - 2}{\sqrt{n (n - 1)}} | H | \normho \right)\\
& \geqslant & \tfrac{n - 1}{n} \left[ n + \tfrac{2}{n} | H |^2 - \left(
\tfrac{1}{n - 1} - \varepsilon \right) | H |^2 - L - \tfrac{n - 2}{\sqrt{n
(n - 1)}} | H | \sqrt{\tfrac{| H |^2}{n (n - 1)} + L} \right]\\
& \geqslant & \tfrac{n - 1}{n} \left[ n + \left( \tfrac{n - 2}{n (n - 1)}
+ \varepsilon \right) | H |^2 - L - (n - 2) \left( \tfrac{| H |^2}{n (n -
1)} + L \right) \right]\\
& \geqslant & \tfrac{n - 1}{n} [\varepsilon | H |^2 - n L] .
\end{eqnarray*}

Let $x$ be a point on $M$ where $| H |$ achieves its maximum. Consider all
the geodesics of length $l = (4 \eta \max_M | H |)^{- 1}$ starting from $x$.
Since $| \nabla | H |^2 | < 4 \eta^2 \max_M | H |^3$, we have $| H |^2 >
\max_M | H |^2 - 4 \eta^2 \max_M | H |^3 \cdot l = (1 - \eta) \max_M | H
|^2$ along such a geodesic. Thus we have $\tmop{Ric} > \tfrac{n - 1}{n}
[\varepsilon | H |^2 - n L] > \tfrac{n - 1}{n} \varepsilon \left(
\frac{1}{2} - \eta \right) \max_M | H |^2 > (n - 1) \pi^2 / l^2$ on such
a geodesic. Then from Myers' theorem, these geodesics can reach any point of
$M$.

Therefore, we obtain $\min_M | H |^2 > (1 - \eta) \max_M | H |^2$ and
$\tmop{diam} M \leqslant 2 l$.
\end{proof}

Now we show that under the assumption of Theorem \ref{theoh}, the mean
curvature flow converges to a point or a totally geodesic submanifold.

\begin{theorem}
\label{Tfin}If $M_0$ satisfies $\normho^2 < W$ and $T$ is finite, then $F_t$
converges to a round point as $t \rightarrow T$.
\end{theorem}

\begin{proof}
If $T$ is finite, we have $\max_{M_t} | h |^2 \rightarrow \infty$ as $t
\rightarrow T$. Let $\underline{H} = \min_{M_t} | H |$, $\overline{H} =
\max_{M_t} | H |$. From the preserved pinching condition, we get
$\overline{H} \rightarrow \infty$ as $t \rightarrow T$.

By Theorem \ref{dH2}, for any $\eta \in \left( 0, \frac{\varepsilon^{1 /
2}}{8 n \pi} \right)$, we have $| \nabla H | < (\eta | H |)^2 + \Psi
(\eta)$. Since $\overline{H} \rightarrow \infty$ as $t \rightarrow T$, there
exists a time $\tau$ depending on $\eta$, such that for $t > \tau$,
$\overline{H}^2 > \Psi (\eta) / \eta^2$. Then we have $| \nabla H | < 2
\eta^2 \overline{H}^2$. Using Lemma \ref{ric}, we obtain $\tmop{diam} M_t
\rightarrow 0$ and $\underline{H} / \overline{H} \rightarrow 1$ as $t
\rightarrow T$.

Now we dilate the metric of the ambient space such that the submanifold
maintains its volume along the flow. Using the same method as in
{\cite{liu2012mean}}, we can prove that the rescaled mean curvature flow
converges to a round sphere as the reparameterized time tends to infinity.
\end{proof}

\begin{theorem}
\label{Tinf}If $M_0$ satisfies $\normho^2 < W$ and $T = \infty$, then $F_t$
converges to a totally geodesic submanifold $\mathbb{C}\mathbb{P}^{n / 2}$
as $t \rightarrow \infty$.
\end{theorem}

\begin{proof}
	Firstly, we prove $| H |^2 < C \mathe^{- \varepsilon t / 8}$ by
	contradiction. Suppose that $\overline{H}^2 \cdot \mathe^{\varepsilon t /
		8}$ is unbounded. Then for any small positive number $\eta$, there exists a
	time $\tau$, such that $\overline{H}^2 (\tau) \cdot \mathe^{\varepsilon \tau
		/ 8} > \Psi (\eta) / \eta^2$. By Theorem \ref{dH2}, we have $| \nabla H | <
	2 \eta^2  \overline{H}^2$ on $M_{\tau}$. By Lemma \ref{ric}, we have
	$\underline{H}^2 (\tau) > (1 - \eta) \overline{H}^2 (\tau) > \frac{1 -
		\eta}{\eta^2} \Psi (\eta) \mathe^{- \varepsilon \tau / 8}$. This together
	with Theorem \ref{dH2} yields $| \nabla H |^2 < (\eta | H |)^4 +
	\frac{\eta^4}{(1 - \eta)^2} \underline{H}^4 (\tau) \cdot \mathe^{\varepsilon
		(\tau - t) / 4}$. From the evolution equation of $| H |^2$, we have
	\begin{equation}
	\dt | H |^2 > \Delta | H |^2 + \frac{1}{n} | H |^4 - \frac{1}{2 n}
	\underline{H}^4 (\tau) \cdot \mathe^{\varepsilon (\tau - t) / 4} .
	\label{dtH2gr}
	\end{equation}
	Using the maximum principle, we get $| H |^2 \geqslant \underline{H}^2
	(\tau)$ if $t \geqslant \tau$. Then (\ref{dtH2gr}) yields $\dt | H |^2 >
	\Delta | H |^2 + \frac{1}{2 n} | H |^4$ for $t \geqslant \tau$. Hence, $| H
	|^2$ blows up in finite time. This contradicts the infinity of $T$.
	Therefore, we obtain $| H |^2 < C \mathe^{- \varepsilon t / 8}$.
	
	From Theorem \ref{sa0h2}, we have $| h |^2 = \normho^2 + \frac{1}{n} | H |^2
	\leqslant C \mathe^{- \varepsilon t / 8}$. Since $| h | \rightarrow 0$ as $t
	\rightarrow \infty$, $M_t$ converges to a totally geodesic submanifold
	$M_{\infty}$ as $t \rightarrow \infty$.

In the case of $q \geqslant n - 4$, since $| h |^2 < \psi(|H|^2)$ is preserved,
thus the flow can't converge to a totally geodesic submanifold. So the
dimension and codimension of $M_{\infty}$ satisfies $q < n - 4$. From the
fact that the totally geodesic submanifolds of $\mathbb{C}\mathbb{P}^m$
are totally real submanifolds $\mathbb{R}\mathbb{P}^n$ and K{\"a}hler
submanifolds $\mathbb{C}\mathbb{P}^{n / 2}$ {\cite{chenby}}, we see that
$M_{\infty}$ must be $\mathbb{C}\mathbb{P}^{n / 2}$.
\end{proof}

Combining the results of Theorems \ref{Tfin} and \ref{Tinf}, we complete the
proof of Theorem \ref{theoh}.

At last, we prove the convergence result under the weakly pinching condition
for the mean curvature flow with $q \geqslant n - 4 \geqslant 2$.

\noindent\textit{Proof of Theorem \ref{theo3}.}
From (\ref{paraUhico}), we have
\begin{eqnarray*}
&  & \frac{1}{2} \left( \dt - \Delta \right) \left( \normho^2 -
\mathring{\psi} \right)\\
& \leqslant & \left( \normho^2 - \mathring{\psi} \right) \left( \normho +
\mathring{\psi} + n + 3 - 2 \varepsilon + \left( \tfrac{1}{n} -
\mathring{\psi}' - \varepsilon \right) | H |^2 + 2 \rho_2 \right)\\
&  & + \mathring{\psi} \left( \mathring{\psi} + \tfrac{1}{n} | H |^2 + n
+ 3 \right) - \mathring{\psi}' \cdot | H |^2 \left( \mathring{\psi} +
\tfrac{1}{n} | H |^2 + n \right) .
\end{eqnarray*}
Using the strong maximum principle, we obtain either $\normho^2 <
\mathring{\psi}$ for some $t > 0$, or $\normho^2  = \mathring{\psi}$ holds for all $t \in
[0, t_0)$.

If $\normho^2 <
\mathring{\psi}$ for some $t > 0$, then it follows from Theorem
\ref{theoh} that $F_t$ converges to a round point.

If $\normho^2  = \mathring{\psi}$ holds for all $t \in [0, t_0)$, we have
\[ \mathring{\psi} \left( \mathring{\psi} + \tfrac{1}{n} | H |^2 + n + 3
\right) - \mathring{\psi}' \cdot | H |^2 \left( \mathring{\psi} +
\tfrac{1}{n} | H |^2 + n \right) = 0. \]
From Lemma \ref{psio} (iii), we get $| H | = 0$. Thus $| h |^2 = \psi (0) =
0$. Therefore, $F_t (M)$ is a totally geodesic submanifold for each $t \in
[0, + \infty)$.
\hfill\qedsymbol\\

\section{Appendix}

For an odd integer $n \geqslant 3$, and a real number $\varepsilon \in [0, 1]$,
we define a function $\varphi_{\varepsilon} : [0, + \infty) \rightarrow
\mathbb{R}$ by
\[ \varphi_{\varepsilon} (x):= d_{\varepsilon} + c_{\varepsilon} x -
\sqrt{b^2 x^2 + 2 a b x + e}, \]
where $a = 2 \sqrt{(n^2 - 4 n + 3) b}, b = \min \left\{ \frac{n - 3}{4 n - 4},
\frac{2 n - 5}{n^2 + n - 2} \right\}, c_{\varepsilon} = b + \tfrac{1}{n - 1 +
	\varepsilon}, d_{\varepsilon} = 2 - 2 \varepsilon + a, e =
\sqrt{\varepsilon}$. We set $\varphi = \varphi_0$.

\begin{lemma}
	The function $\varphi$ satisfies
	\begin{enumerateroman}
		\item $\frac{x}{n - 1} + 2 < \varphi (x) < \frac{x}{n - 1} + n$,
		
		\item $\varphi (x) > \sqrt{2 (n - 3)}$.
	\end{enumerateroman}
\end{lemma}

\begin{proof}
	If $n = 3$, then $a = b = 0$. We have $\varphi (x) = 2 + \tfrac{x}{n - 1}$.
	It's easy to verify the inequalities above.
	
	If $n \geqslant 5$, by direct computations, we get
	\[ \varphi' (x) = c_0 - \frac{b x + a}{\sqrt{x^2 + 2 a x / b}}, \]
	\[ \varphi'' (x) = \frac{a^2}{b (x^2 + 2 a x / b)^{3 / 2}} . \]
	Since $\left( \varphi (x) - \frac{x}{n - 1} \right)'' = \varphi'' (x) > 0$
	and $\lim_{x \rightarrow \infty} \varphi' (x) = \frac{1}{n - 1}$, we have
	$\varphi' (x) < \frac{1}{n - 1}$. Hence we get
	\[ 2 = \lim_{x \rightarrow \infty} \left( \varphi (x) - \frac{x}{n - 1}
	\right) < \varphi (x) - \frac{x}{n - 1} \leqslant \varphi (0) = 2 + a <
	n. \]
	If $n \geqslant 5$, we figure out that
	\[ \min_{x \geqslant 0} \varphi (x) = \varphi \left( \frac{a c_0}{b
		\sqrt{c^2_0 - b^2}} - \frac{a}{b} \right) = d_0 - \frac{a c_0}{b} +
	\frac{a}{b} \sqrt{c_0^2 - b^2} . \]
	If $n = 5$, we have $\min_{x \geqslant 0} \varphi (x) = 4 \sqrt{2} - 2$. If
	$n \geqslant 7$, we have $\min_{x \geqslant 0} \varphi (x) = 2 + 2
	\sqrt{\frac{n - 3}{2 n - 5}} \left( \sqrt{5 n - 8} - \sqrt{n + 2} \right) >
	\sqrt{2 (n - 3)}$.
\end{proof}

Letting $\mathring{\varphi}_{\varepsilon} (x) = \varphi_{\varepsilon} (x) -
\frac{x}{n}$, we have

\begin{lemma}
	There exists a positive constant $\varepsilon_1$ depending on $n$, such that
	for all $\varepsilon \in (0, \varepsilon_1)$, the function
	$\varphi_{\varepsilon}$ satisfies the following inequalities
	\begin{enumerateroman}
		\item $2 x \mathring{\varphi}_{\varepsilon}'' (x) +
		\mathring{\varphi}_{\varepsilon}' (x) < \frac{2 (n - 1)}{n (n + 2)} $,
		
		\item $\mathring{\varphi}_{\varepsilon} (x) (\varphi_{\varepsilon} (x) - n
		+ 3) - x \mathring{\varphi}_{\varepsilon}' (x) (\varphi_{\varepsilon} (x)
		+ n + 3) < 2 (n - 1)$,
		
		\item $\mathring{\varphi}_{\varepsilon} (x) - x
		\mathring{\varphi}_{\varepsilon}' (x)>1$.
	\end{enumerateroman}
\end{lemma}

\begin{proof}
	If $n = 3$, then $a = b = 0$. We have $\varphi_{\varepsilon} (x) = 2 - 2
	\varepsilon + \tfrac{x}{2 + \varepsilon} - \sqrt[4]{\varepsilon}$. It's easy
	to verify these inequalities.
	
	If $n \geqslant 5$, by direct computations, we get
	\[ \mathring{\varphi}_{\varepsilon}' (x) = c_{\varepsilon} - \frac{1}{n} -
	\frac{b^2 x + a b}{\sqrt{b^2 x^2 + 2 a b x + e}}, \]
	\[ \mathring{\varphi}_{\varepsilon}'' (x) = \frac{b^2 (a^2 - e)}{(b^2 x^2 +
		2 a b x + e)^{3 / 2}}, \]
	\[ \mathring{\varphi}_{\varepsilon}''' (x) = - \frac{3 b^3 (a^2 - e) (b x +
		a)}{(b^2 x^2 + 2 a b x + e)^{5 / 2}} . \]
	Then we have
	\begin{eqnarray*}
		2 x \mathring{\varphi}_{\varepsilon}'' (x) +
		\mathring{\varphi}_{\varepsilon}' (x) & = & c_{\varepsilon} - \frac{1}{n}
		- \frac{b^3 x^2 (b x + 3 a) + e b (3 b x + a)}{(b^2 x^2 + 2 a b x + e)^{3
		/ 2}}\\
		& < & c_{\varepsilon} - \frac{1}{n}\\
		& < & \frac{2 (n - 1)}{n (n + 2)} .
	\end{eqnarray*}
	This proves inequality (i).
	
	Setting $f (x) = \mathring{\varphi}_{\varepsilon} (x) (\varphi_{\varepsilon}
	(x) - n + 3) - x \mathring{\varphi}_{\varepsilon}' (x)
	(\varphi_{\varepsilon} (x) + n + 3)$, we get
\begin{eqnarray*}
	f (x) & = & e + d_{\varepsilon} (d_{\varepsilon} + 3 - n) + (2 + a b +
	c_{\varepsilon}  (d_{\varepsilon} - 2 n)) x\\
	&  & - (b^2 x^2 + 2 a b x + e)^{- \frac{1}{2}} \times\\
	&  & [(3 a b (d_{\varepsilon} + 1 - n) + c_{\varepsilon} e) x + b (a
	c_{\varepsilon} + b (d_{\varepsilon} - 2 n)) x^2 + e (2 d_{\varepsilon} + 3
	- n)] .
\end{eqnarray*}
	Then we figure out
	\begin{eqnarray*}
		\lim_{x \rightarrow + \infty} f (x) & = & \frac{a^2 c_{\varepsilon}}{b} +
		d_{\varepsilon} (d_{\varepsilon} + 3 - n) + a (n - 3 - 2 d_{\varepsilon}) +
		e \left( 1 - \frac{c_{\varepsilon}}{b} \right)\\
		& = & 2 (n - 1) + \frac{2 \varepsilon (n^2 - 10 n + 13 + 3 \varepsilon (n -
			3) + 2 \varepsilon^2)}{n - 1 + \varepsilon} + e \left( 1 -
		\frac{c_{\varepsilon}}{b} \right)\\
		& < & 2 (n - 1),
	\end{eqnarray*}
	and
	\begin{eqnarray*}
		f' (x) & = & 2 + a b + c_{\varepsilon} d_{\varepsilon} - 2 c_{\varepsilon} n
		- (b^2 x^2 + 2 a b x + e)^{- 3 / 2} \times\\
		&  & [(3 a^2 b^2 (d_{\varepsilon} + 1 - n) - 3 b e (b - a c_{\varepsilon} +
		b n)) x \\
		&  & + b^2 (a c_{\varepsilon} + b d_{\varepsilon} - 2 n b) x^2 (b x + 3
		a)\\
		&  & + c_{\varepsilon} e^2 + a b e (d_{\varepsilon} - 2 n) ] .
	\end{eqnarray*}
	Then we have
	\[ \lim_{x \rightarrow + \infty} f' (x) = 2 + a b + c_{\varepsilon}
	d_{\varepsilon} - 2 c_{\varepsilon} n - a c_{\varepsilon} - b
	d_{\varepsilon} + 2 b n = 0. \]
	Furthermore, we get
	\[ f'' (x) = [b (2 b + b d_{\varepsilon} - a c_{\varepsilon}) x^2 + (a b (1 -
	n + d_{\varepsilon}) - c_{\varepsilon} e) x - (n + 1) e] \frac{3 b^2 (a^2 -
		e)}{(b^2 x^2 + 2 a b x + e)^{5 / 2}} . \]
	From $b \leqslant \frac{n - 3}{4 n - 4}$, we obtain
	\[ 2 b + b d_{\varepsilon} - a c_{\varepsilon} = (4 - 2 \varepsilon) b -
	\frac{2 \sqrt{(n^2 - 4 n + 3) b}}{n - 1 + \varepsilon} < 0, \]
	and
	\[ 1 - n + d_{\varepsilon} = 3 - n + 2 \sqrt{(n^2 - 4 n + 3) b} - 2
	\varepsilon < 0. \]
	So, we get $f'' (x) < 0$. Then we have $f' (x) > 0$. From this we deduce that
	\[ f (x) < \lim_{x \rightarrow + \infty} f (x) < 2 (n - 1) . \]
	Thus, inequality (ii) is proved.
	
	We have
	\begin{eqnarray*}
		\mathring{\varphi}_{\varepsilon} (x) - x \mathring{\varphi}_{\varepsilon}'
		(x) & = & d_{\varepsilon} - \frac{a b x + e}{\sqrt{b^2 x^2 + 2 a b x +
		e}}\\
		& > & d_{\varepsilon} - \frac{a b x}{\sqrt{b^2 x^2}} -
		\frac{e}{\sqrt{e}}\\
		& = & 2 - 2 \varepsilon - \sqrt[4]{\varepsilon} .
	\end{eqnarray*}
	This implies inequality (iii).
\end{proof}

For each integer $n \geqslant 6$, we define the function $\psi : [0, + \infty)
\rightarrow \mathbb{R}$ by
\[ \psi (x) := \nu + \kappa x - \sqrt{\lambda^2 x^2 + 2 \lambda \mu x +
	\nu^2}, \]
where $\kappa = \lambda + \frac{1}{n - 1}, \lambda = \tfrac{3}{n^3 - 4 n^2 +
	3}, \mu = \nu + \frac{3}{n}, \nu = \frac{9}{n^2 - 3 n - 3}$.

\begin{lemma}
	\label{psi}The function $\psi$ satisfies the following inequalities.
	\begin{enumerateroman}
		\item $\frac{1}{n} \leqslant \psi' (x) < \frac{1}{n - 1}$, $\frac{x}{n}
		\leqslant \psi (x) \leqslant \frac{x}{n - 1}$, and the equalities hold if
		and only $x = 0$,
		
		\item $\max_{x \geqslant 0} (2 x \psi'' (x) + \psi' (x)) < \frac{3}{n +
	8}$,
		
		\item $3 \psi (x) - \frac{3 x}{n} + (\psi (x) - x \psi' (x)) (\psi (x) +
		n) \leqslant 0$, and the equalities hold if
		and only $x = 0$,
		
		\item $0 \leqslant x \psi' (x) - \psi (x) < 2$.
	\end{enumerateroman}
\end{lemma}

\begin{proof}
	Taking derivatives, we get
	\[ \psi' (x) = \kappa - \frac{\lambda^2 x + \lambda \mu}{\sqrt{\lambda^2 x^2
	+ 2 \lambda \mu x + \nu^2}}, \]
	\[ \psi'' (x) = \frac{\lambda^2 (\mu^2 - \nu^2)}{(\lambda^2 x^2 + 2 \lambda
		\mu x + \nu^2)^{3 / 2}}, \]
	\[ \psi''' (x) = - \frac{3 \lambda^3 (\mu^2 - \nu^2) (\lambda x +
		\mu)}{(\lambda^2 x^2 + 2 \lambda \mu x + \nu^2)^{5 / 2}} . \]

	(i) We have $\psi' (0) = \kappa - \frac{\lambda \mu}{\nu} = \frac{1}{n}$ and
	$\lim_{x \rightarrow + \infty} \psi' (x) = \kappa - \lambda = \frac{1}{n -
		1}$. Since $\psi'' (x) > 0$, we get $\frac{1}{n} < \psi' (x) < \frac{1}{n -
		1}$ for $x > 0$. Thus we obtain $\frac{x}{n} < \psi (x) < \frac{x}{n - 1}$
	for $x > 0$.
	
	(ii) Letting $g (x) = 2 x \psi'' (x) + \psi' (x)$, we have
	\[ g (x) = \kappa - \frac{\lambda^4 x^3 + 3 \lambda^3 \mu x^2 + 3 \lambda^2
		\nu^2 x + \lambda \mu \nu^2}{(\lambda^2 x^2 + 2 \lambda \mu x + \nu^2)^{3
	/ 2}}, \]
	and
	\[ g' (x) = 2 x \psi''' (x) + 3 \psi'' (x) = \frac{3 \lambda^2 (\mu^2 -
		\nu^2) (\nu^2 - \lambda^2 x^2)}{(\lambda^2 x^2 + 2 \lambda \mu x +
		\nu^2)^{5 / 2}} . \]
	Thus we get
	\begin{eqnarray*}
		\max_{x \geqslant 0} g (x) & = & g \left( \tfrac{\nu}{\lambda} \right) =
		\kappa - \lambda \sqrt{\frac{2 \nu}{\mu + \nu}}\\
		& = & \frac{n (n - 3) - 3 \sqrt{\frac{6 n}{n^2 + 3 n - 3}}}{n^3 - 4 n^2 +
	3}\\
		& < & \frac{3}{n + 8} .
	\end{eqnarray*}

	(iii) We have
	\begin{eqnarray*}
		&  & 3 \psi (x) - \frac{3 x}{n} + (\psi (x) - x \psi' (x)) (\psi (x) +
		n)\\
		& = & \left[ \nu (n + 3 + 2 \nu) + \left( \lambda \mu + 3 \kappa + \kappa
		\nu - \tfrac{3}{n} \right) x \right]\\
		&  & - \frac{\lambda (\kappa \mu + 3 \lambda + \lambda \nu) x^2 + (\kappa
	\nu^2 + \lambda \mu (n + 6 + 3 \nu)) x + \nu^2 (n + 3 + 2
	\nu)}{\sqrt{\lambda^2 x^2 + 2 \lambda \mu x + \nu^2}} .
	\end{eqnarray*}
	Set
	\begin{eqnarray*}
		h (x) & := & \left[ \nu (n + 3 + 2 \nu) + \left( \lambda \mu + 3
		\kappa + \kappa \nu - \tfrac{3}{n} \right) x \right]^2 (\lambda^2 x^2 + 2
		\lambda \mu x + \nu^2)\\
		&  & - [\lambda (\kappa \mu + 3 \lambda + \lambda \nu) x^2 + (\kappa
		\nu^2 + \lambda \mu (n + 6 + 3 \nu)) x + \nu^2 (n + 3 + 2 \nu)]^2 .
	\end{eqnarray*}
	Now we need to prove $h (x) \leqslant 0$. Since $\kappa = \lambda +
	\frac{1}{n - 1}, \mu = \nu + \frac{3}{n}$, we have
	\[ \lambda \mu + 3 \kappa + \kappa \nu - \tfrac{3}{n} = \kappa \mu + 3
	\lambda + \lambda \nu . \]
	Putting $A = \kappa \mu + 3 \lambda + \lambda \nu, B = n + 3 + 2 \nu, C =
	\kappa \nu^2 + \lambda \mu (n + 6 + 3 \nu)$, we get
	\begin{eqnarray*}
		h (x) & = & (\nu B + A x)^2 (\lambda^2 x^2 + 2 \lambda \mu x + \nu^2) -
		(\lambda A x^2 + C x + \nu^2 B)^2\\
		& = & 2 \lambda A (\mu A + \lambda \nu B - C) x^3\\
		&  & + [\nu^2 (A - \lambda B)^2 + 4 \lambda \mu \nu A B - C^2] x^2\\
		&  & + 2 \nu^2 B (\nu A + \lambda \mu B - C) x.
	\end{eqnarray*}
	By the definitions of $\kappa, \lambda, \mu, \nu$, we have
	\[ \mu A + \lambda \nu B = \nu A + \lambda \mu B = C, \]
	and
	\[ \nu^2 (A - \lambda B)^2 + 4 \lambda \mu \nu A B - C^2 = - \frac{81 (n^3 -
		12 n + 9)^2}{n^2 (n - 1)^2 (n^2 - 3 n - 3)^4} < 0. \]
	Thus, inequality (iii) is proved.
	
	(iv) It follows from inequalities (i) and (iii) that $\psi (x) - x \psi' (x)
	\leqslant 0$. By direct computations, we get
	\begin{eqnarray*}
		x \psi' (x) - \psi (x) & = & - \nu + \frac{\lambda \mu x +
	\nu^2}{\sqrt{\lambda^2 x^2 + 2 \lambda \mu x + \nu^2}}\\
		& \leqslant & - \nu + \frac{\lambda \mu x}{\sqrt{\lambda^2 x^2}} +
		\frac{\nu^2}{\sqrt{\nu^2}}\\
		& = & \mu < 2.
	\end{eqnarray*}
	
\end{proof}

\end{document}